\documentclass[12pt, reqno]{article}
\usepackage{enumerate,amsmath,amssymb,bm,ascmac,amsthm,url}
\usepackage{tikz}
\usepackage[margin=30truemm]{geometry}
\usepackage{here}
\usepackage{comment}
\usetikzlibrary{calc}

\theoremstyle{definition}
\newtheorem{thm}{Theorem}[section]
\newtheorem{thm*}{Theorem}
\newtheorem{defi}[thm]{Definition}
\newtheorem*{defi*}{Definition}
\newtheorem{lem}[thm]{Lemma}
\newtheorem{lem*}{Lemma}
\newtheorem{pro}[thm]{Proposition}
\newtheorem{pro*}{Proposition}

\newtheorem{cor}[thm]{Corollary}

\newcommand{\MC}[1]{\mathcal{#1}}
\newcommand{\MB}[1]{\mathbb{#1}}

\newcommand{\G}{\Gamma}

\DeclareMathOperator{\Spec}{Spec}
\DeclareMathOperator{\re}{Re}

\makeatletter

\@addtoreset{equation}{section}
\makeatother

\allowdisplaybreaks

\title
{
On symmetric spectra of Hermitian adjacency matrices for non-bipartite mixed graphs
}

\author{Yusuke Higuchi\thanks{
Department of Mathematics,  
Faculty of Science, Gakushuin University, Tokyo 171-8588, Japan.
\texttt{higuchi@math.gakushuin.ac.jp}
}
\and 
Sho Kubota\thanks{
Department of Applied Mathematics, Faculty of Engineering, Yokohama National University, 
Hodogaya, Yokohama 240-8501, Japan. \texttt{kubota-sho-bp@ynu.ac.jp}
}
\and
Etsuo Segawa\thanks{
Graduate School of Environment Information Sciences, Yokohama National University, 
Hodogaya, Yokohama 240-8501 Japan. \texttt{segawa-etsuo-tb@ynu.ac.jp}
}
}

\date{}

\begin{document}
\maketitle
\begin{abstract}
We study the equivalence between bipartiteness and symmetry of spectra of mixed graphs,
for $\theta$-Hermitian adjacency matrices defined by an angle $\theta \in (0, \pi]$.
We show that this equivalence holds when, for example, an angle $\theta$ is an algebraic number,
while it breaks down for any angle $\theta \in \MB{Q}\pi$.
Furthermore, we construct a family of non-bipartite mixed graphs having the symmetric spectra for given $\theta \in \MB{Q}\pi$.
\vspace{8pt} \\
{\it Keywords:} Symmetric spectrum, Hermitian adjacency matrix, mixed graph \\
{\it MSC 2020 subject classifications:} 05C50; 05C20
\end{abstract}

\section{Introduction}\label{Intro}

A {\it mixed graph} $G$ consists of a finite set $V$ of vertices
together with a subset $\MC{A} \subset V \times V \setminus \{ (x,x) \mid x \in V \}$
of ordered pairs called {\it arcs}.
Let $G = (V, \MC{A})$ be a mixed graph.
For $(x,y) \in \MC{A}$,
we denote $x \to y$ if $(x,y) \in \MC{A}$ but $(y,x) \not\in \MC{A}$.
The same is for $x \gets y$.
The symbol $x \leftrightarrow y$ means $(x,y) \in \MC{A}$ and $(y,x) \in \MC{A}$.
In this case, the unordered pair $\{x,y\}$ is called a {\it digon}.
A digon is equated with an undirected edge depending on the context.
The {\it underlying graph} of a mixed graph $G = ( V, \MC{A})$, denoted by $\G(G)$,
is the graph with the vertex set $V$ and the edge set $E = \{ \{x,y\} \mid (x,y) \in \MC{A} \text{ or } (y,x) \in \MC{A} \}$.
We say that a mixed graph with no digons is an {\it oriented graph},
and a mixed graph with at least one digon is {\it proper}.
For $\theta \in (0, \pi]$, 
the {\it $\theta$-Hermitian adjacency matrix} $H_{\theta} = H_{\theta}(G) \in \MB{C}^{V \times V}$ is defined by
\begin{equation}\label{hermit}
(H_{\theta})_{x,y} = \begin{cases}
1 \qquad &\text{if $x \leftrightarrow y$,} \\
e^{i \theta} \qquad &\text{if $x \to y$,} \\
e^{-i \theta} \qquad &\text{if $x \gets y$,} \\
0 \qquad &\text{otherwise,}
\end{cases}
\end{equation}
where $i$ is the imaginary unit.
The matrix $H_{\frac{\pi}{2}}$ is so-called the Hermitian adjacency matrix (of the first kind),
introduced by Liu--Li \cite{LL} and Guo--Mohar \cite{GM} independently.
The matrix $H_{\frac{\pi}{3}}$ is called the Hermitian adjacency matrix of the second kind,
introduced by Mohar \cite{M}.
Very recently, Yu et al \cite{YGZ} introduced $k$-generalized Hermitian adjacency matrices,
which correspond to the case $\theta = \frac{2\pi}{k}$,
and mixed graphs with their small ranks were studied.
The case $\theta = \pi$ can be regarded as the adjacency matrix of a signed graph.
For a general angle $\theta$, the authors in \cite{KST} introduced $H_{\theta}$
in the context of quantum walks,
and identification of mixed graphs by their spectra has been attempted.
Note that $H_{\theta}(\G(G))$
and the ordinary adjacency matrix of $\G(G)$ are equal for any angle $\theta$.
The multiset of eigenvalues of $H_{\theta}(G)$ is called {\it $H_{\theta}$-spectrum} of $G$
and denoted by $\Spec(G; \theta)$.
We say that $G$ has the {\it $\theta$-symmetric spectrum}
if the elements of $\Spec(G; \theta)$ are symmetric with respect to the origin, including their multiplicities.
Here we should remark that the Hermitian adjacency matrices stated above 
have been essentially introduced in more general forms: 
for example, as a perturbed operator on a graph, 
the properties of the discrete magnetic/twisted Laplacian and 
the twisted quantum walk are discussed in \cite{HS,KSS,LiLo,Su}  and \cite{HKSS}, 
respectively.
We give a short review for the magnetic/twisted adjacency operator 
in Section~\ref{MAM}. 
However, in this paper, 
we follow the setting and notation of the Hermitian adjacency matrix 
discussed in \cite{GM,LL,M}.

A mixed graph $G$ is said to be {\it bipartite} if its underlying graph $\G(G)$ is bipartite.
As is well known,
an undirected graph is bipartite if and only if it has the symmetric spectrum
in the sense of adjacency matrices.
This fact is derived from a particularity of real and non-negative matrices.
Therefore for $\theta$-Hermitian adjacency matrices in general,
the equivalence between bipartiteness and having the $\theta$-symmetric spectrum can be broken.
Indeed, as Liu--Li points out \cite{LL}, 
the $H_{\frac{\pi}{2}}$-spectrum of any oriented graph is always symmetric with respect to the origin.
For a mixed graph having digons,
Guo--Mohar \cite{GM} provided an example of a mixed graph
in which the equivalence breaks down for $\theta = \frac{\pi}{2}$.
The mixed graph provided by them is shown in Figure~\ref{0923-1}.
Mohar \cite{M} provides a mixed graph with 5 vertices that is non-bipartite
but has the $\frac{\pi}{3}$-symmetric spectrum.
In signed graphs, which correspond to $\theta = \pi$ in our setting,
a class called sign-symmetric has been studied,
so signed graphs with symmetric spectra have been substantially studied.
As seen in \cite{GHMM, S},
families of non-bipartite signed graphs but have symmetric spectra have been found.

\begin{figure}[ht]
\begin{center}
\begin{tikzpicture}
[scale = 0.7,
line width = 0.8pt,
v/.style = {circle, fill = black, inner sep = 0.8mm},u/.style = {circle, fill = white, inner sep = 0.1mm}]
  \node[v] (1) at (0, 0) {};
  \node[v] (2) at (2, 0) {};
  \node[v] (3) at (2, 2) {};
  \node[v] (4) at (0, 2) {};
  \draw[-] (1) to (2);
  \draw[->] (3) to (2); 
  \draw[->] (4) to (3);
  \draw[->] (4) to (1);
  \draw[->] (3) to (1);
  \draw[->] (2) to (4);
\end{tikzpicture}
\caption{The non-bipartite mixed graph provided by Guo--Mohar \cite{GM}
that has the $\frac{\pi}{2}$-symmetric spectrum} \label{0923-1}
\end{center}
\end{figure}
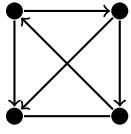

As we have seen sporadically in the above literatures,
we know examples of non-bipartite mixed graphs
having $\theta$-symmetric spectra for $\theta \in \{ \frac{\pi}{2}, \frac{\pi}{3}, \pi \}$.
Then, is there always non-bipartite mixed graphs having the $\theta$-symmetric spectrum for each $\theta$?
We are interested in this question.
Therefore, we define the following terminology for angles.
\begin{defi}[Bipartite Detection Property]
Let us fix  an angle $\theta \in (0, \pi]$ and consider any mixed graph $G$. 
$\theta$ is said to have the {\it bipartite detection property}
when the following equivalence holds: 
$G$ is bipartite if and only if $G$ has the $\theta$-symmetric spectrum.
Here, we recall that a mixed graph $G$ has the {\it $\theta$-symmetric spectrum} if the elements of $\Spec(G; \theta)$ are symmetric with respect to the origin, including their multiplicities.
\end{defi}
Based on this terminology,
we know so far that the three angles $\frac{\pi}{2}, \frac{\pi}{3}$ and $\pi$
do not have the bipartite detection property.
We are interested in which angle has the bipartite detection property.
In conclusion, we found the following:


\begin{thm} \label{main_BDP}
{\it
Let $\theta \in (0, \pi]$.
If an angle $\theta$ satisfies one of the following, then it has the bipartite detection property:
\begin{enumerate}[(i)]
\item $\theta$ is an algebraic number.
\item $\theta$ is described by $\alpha \pi$ for an irrational algebraic number $\alpha$.
\item $\theta$ is described by $\alpha \pi + \beta$ for non-zero algebraic numbers $\alpha, \beta$.
\end{enumerate}
}
\end{thm}

\begin{thm} \label{main_notBDP}
{\it
Let $\theta \in (0, \pi]$.
Any angle $\theta \in \MB{Q}\pi$ does not have the bipartite detection property.
}
\end{thm}

Since the set of real algebraic numbers is dense in $\MB{R}$, 
it turns out, 
as a simple corollary of our theorems, 
that a bipartite graph only has the ``persistent'' $\theta$-symmetric 
spectra for the perturbation on $\theta$ as follows:
\begin{cor}\label{Cor1_2}
{\it
Fix any $\theta_{0} \in (0, \pi]$. 
A graph $G$ is bipartite if and only if 
there exists a positive number $\epsilon>0$ such that 
$G$ has the $\theta$-symmetric spectrum for 
any $\theta \in (\theta_{0}-\epsilon, \theta_{0}+\epsilon)$. 
}
\end{cor}
Naturally, the above shows 
a kind of the identity theorem, that is,  
$G$ has the $\theta$-symmetric spectrum for any $\theta\in \MB{R}$  if there exists a positive number $\epsilon>0$ such that 
$G$ has the $\theta$-symmetric spectrum for 
any $\theta \in (\theta_{0}-\epsilon, \theta_{0}+\epsilon)$. 

This paper is organized as follows.
In Section~\ref{MAM},
we describe relationship between Hermitian adjacency matrices introduced in spectral graph theory and magnetic adjacency operators introduced in spectral analysis on graphs.
In Section~\ref{BT},
basic tools for discussing bipartite detection property are organized.
In Section~\ref{SA},
we prove our first main theorem (Theorem~\ref{main_BDP}), that is,
we present three sufficient conditions for an angle $\theta \in (0, \pi]$ to have the bipartite detection property.
In Section~\ref{1117-1},
we prove our second main theorem (Theorem~\ref{main_notBDP}).
Namely, when an angle $\theta \in (0, \pi]$ is in $\MB{Q}\pi$,
a non-bipartite oriented graph which has the $\theta$-symmetric spectrum is constructed depending on the angle.
In Section~\ref{1222-1}, certain minimality of the oriented graphs constructed in the previous section is discussed.

\section{Magnetic Adjacency Matrix}\label{MAM}

In this section, we give a short review for the magnetic adjacency operator
 interpreted from the notation and results for the magnetic/twisted Laplacian.
For details, one should refer to \cite{HS,KSS,LiLo,Su}. 
Here, we follow the notation in \cite{HS}. 

Let $G= (V(G),E(G))$ be a connected, locally finite and unoriented  graph, 
where $V(G)$ is the set of vertices and $E(G)$ is the set of its unoriented 
edges. A graph $G$ may be infinite, and have self-loops and multiple edges. 
Considering each edge in $E(G)$ to have two orientations, 
we introduce the set of all oriented edges, which is denoted by $A(G)$. 
For an edge $e\in A(G)$, the origin and the terminus of $e$ are denoted by 
$o(e)$ and $t(e)$, respectively. Moreover the inverse edge of $e$ is 
denoted by $\bar{e}$. 

Putting 
$$
C^{1}_{\mathbb{R}}(G) = \{\omega\, :\, A(G)\to\mathbb{R}\ ;\  
\omega(\bar{e})=-\omega(e) \}, 
$$
we call an element of $C^{1}_{\mathbb{R}}(G)$ a $1$-form on $G$. 
For a fixed $1$-form $\omega$, we define a self-adjoint operator 
$H_{\omega,G}\, :\, \ell^2(V(G))\to\ell^{2}(V(G))$ by
\begin{equation}
H_{\omega,G} f(x)= \sum_{e\in A_{x}(G)}\exp(i\omega(e))f(
(t(e)),
\end{equation}
where $A_{x}(G)=\{e\in A(G)\ ;\ o(e)=x\}$ and 
$\ell^{2}(V(G)) $ is the Hilbert space with the
standard inner product $\langle\cdot , \cdot\rangle$, that is, 
$$
\ell^{2}(V(G)) = \{f\, :\, V(G)\to\mathbb{C}\ ;\ \langle f,f\rangle<\infty\}.
$$
In this setting, the $\theta$-Hermitian adjacency matrix $H_{\theta}$ 
on a mixed graph $G$ introduced in Section~\ref{Intro} can 
be interpreted as the magnetic adjacency operator on a finite graph $G$ 
for a special $1$-form $\omega$. 
For a mixed graph $G=(V,\MC{A})$, let us first choose $A_{0}\subset A(G)$ following 
the set of arcs of the mixed graph $G=(V,\MC{A})$. 
Then we set the $1$-form $\omega$ on 
$A(G)$ as (i)~$\omega(e)=\theta$ for $e\in A_{0}$ (therefore, 
$\omega(\bar{e})=-\theta$)  and (ii)~$\omega(e)=0$ for 
$e\in A(G)\setminus (A_{0}\cup\overline{A_{0}})$, 
where $\overline{A_{0}}=\{e\in A(G)\ ;\ \bar{e}\in A(G)\}$. 
In this sense,  the $\theta$-Hermitian adjacency matrix $H_{\theta}$ 
is essentially the same as the magnetic/twisted Laplacian discussed in 
\cite{HS,KSS,LiLo,Su}.

Now let us give some important properties for a magnetic adjacency matrix. 
We can immediately transplant in terms of a magnetic adjacency matrix
from the known results (cf. \cite{HS,KSS,LiLo,Su})
in the context of the magnetic/twisted Laplacian.
Here we write $Spec(H_{\omega, G})$ for the set of eigenvalues of $H_{\omega, G}$ 
including their multiplicities.

\begin{thm}[cf.\cite{HS}]
For a given $1$-form $\omega\in C^{1}_{\mathbb{R}}(G)$, 
put $\omega'(e)=-\omega(e)$ for every $e\in A(G)$. 
Then $Spec(H_{\omega, G}) = Spec(H_{\omega', G})$. 
\end{thm}

\begin{thm}[cf.\cite{HS}]
For a given $1$-form $\omega\in C^{1}_{\mathbb{R}}(G)$, 
set a $1$-form $\omega'$ as follows: 
$\omega'(e)=\pi -\omega(e)$ 
for every $e\in A(G)$. 
Remark that $\omega'(e)=-\omega'(\bar{e})$ in modulo $2\pi$ since 
$\omega(e)=-\omega(\bar{e})$. 
Then $H_{\omega,G}$ and $-H_{\omega',G}$ are unitarily equivalent, that is, 
$Spec(H_{\omega,G})$ and $Spec(H_{\omega',G})$ are mutually symmetric with 
respect to $0$. 
\end{thm}
The following, which is also found in \cite{LL},
can be immediately derived from the fact above: 
if a $1$-form $\omega$ satisfies $|\omega(e)|=\pi/2$ for 
every $e\in A(G)$,
then $Spec(H_{\omega',G})$ is symmetric with respect to $0$.  

The next lemma is a basic result for a general magnetic adjacency matrix.
\begin{lem}[cf.\cite{HS}]\label{basiclemma}
Suppose that a graph $G$ is bipartite. 
Then $H_{\omega,G}$ and $-H_{\omega,G}$ are unitarily equivalent
for any $1$-form $\omega\in C^{1}_{\mathbb{R}}(G)$. In other words,
$Spec(H_{\omega,G})$ is symmetric with respect to $0$. 
\end{lem}

If a path $C=(e_{1},e_{2},\dots ,e_{n})$, which is a sequence of oriented edges
such that $t(e_{k})=o(e_{k+1})$, satisfies $t(e_{n})=o(e_{1})$, then
it is said to be a closed path. For a closed path $C$, we define
the {\it magnetic flux\/} of a $1$-form $\omega\in C^{1}_{\mathbb{R}}(G)$
through $C$ by
\begin{equation}
\int_{C}\omega=\sum_{k=1}^{n}\omega(e_{k}).
\end{equation} 

The next statement shows the magnetic fluxes of
$1$-form $\omega\in C^{1}_{\mathbb{R}}(G)$
over $G$ determine the spectrum of $H_{\omega, G}$ (cf. \cite{HS,LiLo}). 
For a finite graph $G$, 
this can be expressed in terms of the characteristic polynomial of 
$H_{\omega, G}$, which is found in Theorem~\ref{0923-4}
in Section~\ref{BT} and also in \cite{MKS}. 

\begin{thm}\label{UE}[cf.\cite{HS,LiLo}]
Let $\omega_{1},\omega_{2}\in C^{1}_{\mathbb{R}}(G)$ be $1$-forms on $G$.
If the magnetic flux of $\omega_{1}$ equals to that of $\omega_{2}$
in modulo $2\pi$ for each closed path of $G$, then
$H_{\omega_{1},G}$ and $H_{\omega_{2},G}$ are unitarily equivalent. that is,
$Spec(H_{\omega_{1},G}) = Spec(H_{\omega_{2},G})$. 
\end{thm} 

One of the applications of the magnetic/twisted operator is that 
the long time behavior of random walks and quantum walks
on the crystal lattice 
can be effectively characterized in \cite{KSS} and \cite{HKSS}, respectively.

Now we close the review for previous research on the magnetic/twisted operator.
Hereinafter, the notation in this paper returns to that introduced in Section~\ref{Intro}. 

\section{Basic tools}\label{BT}

Thereafter, we return to the discussion in spectral graph theory on finite graphs, described by matrices.
This section discusses basic tools for studying symmetric spectra.
First, the following is a special case for Lemma~\ref{basiclemma},
which can be also proved by extending the procedure performed for adjacency matrices of undirected bipartite graphs:
\begin{lem} \label{0926-1}
{\it
If a mixed graph $G$ is bipartite, then it has the $\theta$-symmetric spectrum for any $\theta \in (0, \pi]$.
}\qed
\end{lem}

%

Fix an angle $\theta$.
Let $G$ be a mixed graph.
We say that a subgraph $C$ of $G$ is a {\it cycle} when $\G(C)$ is a cycle in $\G(G)$.
Let $C$ be a cycle in $\G(G)$.
Display as $C = (x_1, x_2, \dots, x_t, x_1)$.
Then the real part of $(H_{\theta})_{x_1, x_2} (H_{\theta})_{x_2, x_3} \cdots (H_{\theta})_{x_t, x_1}$
does not depend on the two directions of the cycle,
i.e.,
$\re ((H_{\theta})_{x_1, x_2} (H_{\theta})_{x_2, x_3} \cdots (H_{\theta})_{x_t, x_1}) = \re ((H_{\theta})_{x_1, x_t} (H_{\theta})_{x_t, x_{t-1}} \cdots (H_{\theta})_{x_2, x_1})$.
We denote this value by $\re(C)$.
An {\it elementary subgraph} $H$ of a mixed graph $G$ is a subgraph of
$G$ such that its connected components consist only of $K_2$ and cycles.
We denote the set of all elementary subgraphs on $j$ vertices in a mixed graph $G$ by $\MC{H}_j(G)$.
Elementary subgraphs dominate the coefficients of the characteristic polynomials of $H_{\theta}$.
The following is claimed by Mehatari et al in the context of $\MB{T}$-gain graphs \cite{MKS}.
Symbols are adjusted to our setting.

\begin{thm}[Corollary~3.1 in \cite{MKS}] \label{0923-4}
{\it
Let $G$ be a mixed graph with $n$ vertices,
and let $\theta \in (0, \pi]$.
Display as $\det(xI_n - H_{\theta}(G)) = \sum_{j = 0}^n a_j x^{n-j}$.
Then we have
\begin{equation} \label{0921-1}
a_j = \sum_{H \in \MC{H}_j(G)} (-1)^{p(H)}2^{|\MC{C}(H)|} \prod_{C \in \MC{C}(H)} \re(C),
\end{equation}
where $p(H)$ is the number of components in $\G(H)$ and $\MC{C}(H)$ is the set of cycles in $H$.
}
\end{thm}

Note that Equality~(\ref{0921-1}) should be precisely described as
\[
a_j = \begin{cases}
\sum_{H \in \MC{H}_j(G)} (-1)^{p(H)}2^{|\MC{C}(H)|} \prod_{C \in \MC{C}(H)} \re(C) \quad &\text{if $\MC{H}_j(G) \neq \emptyset$ and $\MC{C}(H) \neq \emptyset$,} \\
\sum_{H \in \MC{H}_j(G)} (-1)^{p(H)} \quad &\text{if $\MC{H}_j(G) \neq \emptyset$ and $\MC{C}(H) = \emptyset$,} \\
0 \quad &\text{if $\MC{H}_j(G) = \emptyset$.}
\end{cases}
\]
There is no elementary subgraph with $1$ vertex, so
\begin{equation} \label{0923-2}
a_1 = 0.
\end{equation}
Furthermore, an elementary subgraph with $2$ vertices corresponds to an edge of the underlying graph,
hence
\begin{equation} \label{0923-3}
a_2 = -|E(\G(G))|.
\end{equation}
For observation,
let us compute the characteristic polynomial of $H_{\theta}(G)$
for the mixed graph $G$ in Figure~\ref{0923-1} of Section~\ref{Intro}.
Let $\det(xI_4 - H_{\theta}(G)) = x^4 + a_1x^3 + a_2x^2 + a_3x + a_4$.
By Equalities~(\ref{0923-2}) and~(\ref{0923-3}), we have $a_1 = 0$ and $a_2 = -6$.
To compute $a_3$ and $a_4$,
we find all elementary subgraphs with $3$ and $4$ vertices.
The elementary subgraphs in $\MC{H}_3(G)$ and $\MC{H}_4(G)$
are shown in Figure~\ref{0923-H3} and Figure~\ref{0923-H4}, respectively.
For each elementary subgraph,
computing $(-1)^{p(H)}2^{|\MC{C(H)|}} \prod_{C \in \MC{C}(H)} \re(C)$ in Theorem~\ref{0923-4}
yields $a_3 = -2(1 + \cos \theta + \cos 2\theta + \cos 3\theta)$
and $a_4 = 3 - 2(\cos \theta + \cos 2\theta + \cos 3\theta)$.
In particular, we have $\det(xI_4 - H_{\theta}(G)) = x^4 - 6x^2 + 5$ for $\theta = \frac{\pi}{2}$.
Thus, $\Spec(G; \frac{\pi}{2}) = \{ \pm \sqrt{5}, \pm 1 \}$,
so we see that $G$ has the $\frac{\pi}{2}$-symmetric spectrum even though it is non-bipartite.

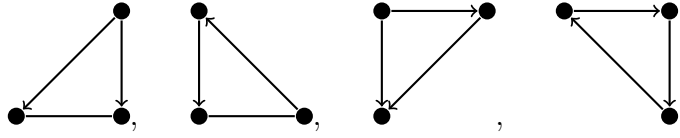
\begin{figure}[ht]
\begin{center}
\begin{tikzpicture}
[scale = 0.7,
line width = 0.8pt,
v/.style = {circle, fill = black, inner sep = 0.8mm},u/.style = {circle, fill = white, inner sep = 0.1mm}]
  \node[v] (1) at (0, 0) {};
  \node[v] (2) at (2, 0) {};
  \node[v] (3) at (2, 2) {};
  \draw[-] (1) to (2);
  \draw[->] (3) to (2); 
  \draw[->] (3) to (1);
\end{tikzpicture},
$\quad$
\begin{tikzpicture}
[scale = 0.7,
line width = 0.8pt,
v/.style = {circle, fill = black, inner sep = 0.8mm},u/.style = {circle, fill = white, inner sep = 0.1mm}]
  \node[v] (1) at (0, 0) {};
  \node[v] (2) at (2, 0) {};
  \node[v] (4) at (0, 2) {};
  \draw[-] (1) to (2);
  \draw[->] (4) to (1);
  \draw[->] (2) to (4);
\end{tikzpicture},
$\quad$
\begin{tikzpicture}
[scale = 0.7,
line width = 0.8pt,
v/.style = {circle, fill = black, inner sep = 0.8mm},u/.style = {circle, fill = white, inner sep = 0.1mm}]
  \node[v] (1) at (0, 0) {};
  \node[v] (3) at (2, 2) {};
  \node[v] (4) at (0, 2) {};
  \draw[->] (4) to (3);
  \draw[->] (4) to (1);
  \draw[->] (3) to (1);
\end{tikzpicture},
$\quad$
\begin{tikzpicture}
[scale = 0.7,
line width = 0.8pt,
v/.style = {circle, fill = black, inner sep = 0.8mm},u/.style = {circle, fill = white, inner sep = 0.1mm}]
  \node[v] (2) at (2, 0) {};
  \node[v] (3) at (2, 2) {};
  \node[v] (4) at (0, 2) {};
  \draw[->] (3) to (2); 
  \draw[->] (4) to (3);
  \draw[->] (2) to (4);
\end{tikzpicture}
\caption{All elementary subgraphs of $G$ with $3$ vertices} \label{0923-H3}
\end{center}
\end{figure}

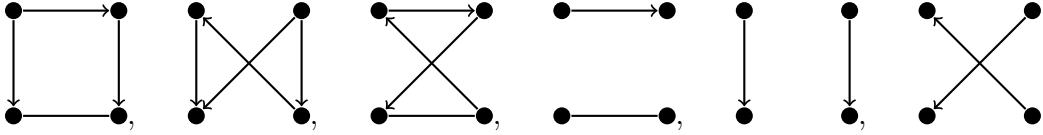
\begin{figure}[ht]
\begin{center}
\begin{tikzpicture}
[scale = 0.7,
line width = 0.8pt,
v/.style = {circle, fill = black, inner sep = 0.8mm},u/.style = {circle, fill = white, inner sep = 0.1mm}]
  \node[v] (1) at (0, 0) {};
  \node[v] (2) at (2, 0) {};
  \node[v] (3) at (2, 2) {};
  \node[v] (4) at (0, 2) {};
  \draw[-] (1) to (2);
  \draw[->] (3) to (2); 
  \draw[->] (4) to (3);
  \draw[->] (4) to (1);
\end{tikzpicture},
$\quad$
\begin{tikzpicture}
[scale = 0.7,
line width = 0.8pt,
v/.style = {circle, fill = black, inner sep = 0.8mm},u/.style = {circle, fill = white, inner sep = 0.1mm}]
  \node[v] (1) at (0, 0) {};
  \node[v] (2) at (2, 0) {};
  \node[v] (3) at (2, 2) {};
  \node[v] (4) at (0, 2) {};
  \draw[->] (3) to (2); 
  \draw[->] (4) to (1);
  \draw[->] (3) to (1);
  \draw[->] (2) to (4);
\end{tikzpicture},
$\quad$
\begin{tikzpicture}
[scale = 0.7,
line width = 0.8pt,
v/.style = {circle, fill = black, inner sep = 0.8mm},u/.style = {circle, fill = white, inner sep = 0.1mm}]
  \node[v] (1) at (0, 0) {};
  \node[v] (2) at (2, 0) {};
  \node[v] (3) at (2, 2) {};
  \node[v] (4) at (0, 2) {};
  \draw[-] (1) to (2);
  \draw[->] (4) to (3);
  \draw[->] (3) to (1);
  \draw[->] (2) to (4);
\end{tikzpicture},
$\quad$
\begin{tikzpicture}
[scale = 0.7,
line width = 0.8pt,
v/.style = {circle, fill = black, inner sep = 0.8mm},u/.style = {circle, fill = white, inner sep = 0.1mm}]
  \node[v] (1) at (0, 0) {};
  \node[v] (2) at (2, 0) {};
  \node[v] (3) at (2, 2) {};
  \node[v] (4) at (0, 2) {};
  \draw[-] (1) to (2);
  \draw[->] (4) to (3);
\end{tikzpicture},
$\quad$
\begin{tikzpicture}
[scale = 0.7,
line width = 0.8pt,
v/.style = {circle, fill = black, inner sep = 0.8mm},u/.style = {circle, fill = white, inner sep = 0.1mm}]
  \node[v] (1) at (0, 0) {};
  \node[v] (2) at (2, 0) {};
  \node[v] (3) at (2, 2) {};
  \node[v] (4) at (0, 2) {};
  \draw[->] (3) to (2); 
  \draw[->] (4) to (1);
\end{tikzpicture},
$\quad$
\begin{tikzpicture}
[scale = 0.7,
line width = 0.8pt,
v/.style = {circle, fill = black, inner sep = 0.8mm},u/.style = {circle, fill = white, inner sep = 0.1mm}]
  \node[v] (1) at (0, 0) {};
  \node[v] (2) at (2, 0) {};
  \node[v] (3) at (2, 2) {};
  \node[v] (4) at (0, 2) {};
  \draw[->] (3) to (1);
  \draw[->] (2) to (4);
\end{tikzpicture}
\caption{All elementary subgraphs of $G$ with $4$ vertices} \label{0923-H4}
\end{center}
\end{figure}

As can be seen from the above observation,
we have
\[ (-1)^{p(H)}2^{|\MC{C(H)|}} \prod_{C \in \MC{C}(H)} \re(C)
\in \MB{Z}[\cos \theta, \cos 2\theta, \dots, \cos j \theta] \]
for any elementary subgraph $H$ with $j$ vertices,
where
\[ \MB{Z}[\alpha_1, \dots, \alpha_s] = \{ f(\alpha_1, \dots, \alpha_s) \mid f \in \MB{Z}[x_1, \dots, x_s] \}. \]
On the other hand,
the trigonometric addition formula derives $\MB{Z}[\cos \theta, \cos 2\theta, \dots, \cos j \theta] = \MB{Z}[\cos \theta]$,
so we obtain
\begin{equation} \label{0925-1}
a_j \in \MB{Z}[\cos \theta]
\end{equation}
 for $j = 1,2, \dots, n$.
We focus on the coefficients of the characteristic polynomial of $\theta$-Hermitian adjacency matrix
to investigate symmetry of the spectrum of a mixed graph.
The following is intuitively almost obvious but useful.

\begin{lem} \label{0926-3}
{\it
Let $G$ be a mixed graph with $n$ vertices,
and let $\Phi(x) = \sum_{j = 0}^n a_j x^{n-j}$ be the characteristic polynomial of $H_{\theta}(G)$ for $\theta \in (0, \pi]$.
Then, $G$ has the $\theta$-symmetric spectrum if and only if $a_{2l-1} = 0$ for any $l$.
}
\end{lem}

\begin{proof}
First, we suppose that $G$ has the $\theta$-symmetric spectrum.
Let $m$ be the multiplicity of $0 \in \Spec(G; \theta)$.
Since $G$ has the $\theta$-symmetric spectrum,
$\alpha \in \Spec(G; \theta) \setminus \{0\}$ if and only if $-\alpha \in \Spec(G; \theta) \setminus \{0\}$ including multiplicity.
Thus, the characteristic polynomial $\Phi(x)$ can be displayed as
$\Phi(x) = x^m(x^2 - \alpha_1^2)(x^2 - \alpha_2^2) \cdots (x^2 - \alpha_t^2)$ for $t$ such that $m+2t = n$.
Then, there exist $s_1, \dots, s_t \in \MB{R}$ such that
$\Phi(x) = x^m(x^{2t}+s_1x^{2t-2} + \cdots + s_{t-1}x^2 + s_t) = \sum_{j = 0}^t s_j x^{n - 2j}$.
In particular,
the coefficients except $x^n, x^{n-2}, \dots, x^{n-2t}$ are zero.
This implies that $a_{2l-1} = 0$ for any $l$.

Next, we suppose that $a_{2l-1} = 0$ for any $l$.
If $n$ is even, we can set $n = 2m$.
We have $\Phi(x) = \sum_{j = 0}^m a_{2j}x^{2m-2j} = \sum_{j = 0}^m a_{2j}(x^2)^{m-j}$.
Then, let $\Psi(u)$ be the polynomial by putting $u = x^2$ in $\Phi(x)$.
For a root $\alpha$ of the polynomial $\Psi(u)$,
the real numbers $\pm \sqrt{\alpha}$ are roots of the characteristic polynomial $\Phi(x)$.
Therefore, $G$ has the $\theta$-symmetric spectrum.
When $n$ is odd, the constant term of $\Phi(x)$ is zero.
Thus, applying the discussion for the case where $n$ is even
to the polynomial $\frac{\Phi(x)}{x}$ completes the proof.
\end{proof}

\section{Strange angles} \label{SA}
In this section, we show that ``strange" angles such as $\theta = 1$,
$\theta = \frac{\sqrt{2}}{2}\pi$, and $\theta = \sqrt{2}\pi - 3$,
which are not usually used as an angle, have the bipartite detection property. 
More precisely,
we show that angles whose cosines are not algebraic numbers have the bipartite detection property.
For a mixed graph $G$,
the coefficients of the characteristic polynomial of $H_{\theta}(G)$
can be viewed as polynomials of $\cos \theta$ as we saw in~(\ref{0925-1}) i.e.,
there exists $f_j \in \MB{Z}[x]$ such that $a_j = f_j(\cos \theta)$ for any $j$.
In studying the bipartite detection property,
the key is to check whether the polynomial $f_j$ is the zero polynomial or not.

\begin{lem} \label{0926-2}
{\it
Let $G$ be a mixed graph with $n$ vertices,
and let $\Phi(x) = \sum_{j = 0}^n a_j x^{n-j}$ be the characteristic polynomial of $H_{\theta}(G)$ for $\theta \in (0, \pi]$.
If $G$ is non-bipartite,
then there exist a positive integer $l$ and a polynomial $f_{2l-1} \in \MB{Z}[x] \setminus \{0\}$
such that $a_{2l-1} = f_{2l-1}(\cos \theta)$.
}
\end{lem}

\begin{proof}
Since $G$ is non-bipartite,
there exists an odd cycle in $G$.
We denote the odd cycles in $G$ with the smallest length as $C_1, C_2, \dots, C_s$,
and let the length of the odd cycles be $2l-1$.
The elementary subgraphs of $G$ with $2l-1$ vertices consist only of connected cycles of length $2l-1$,
i.e., $\MC{H}_{2l-1}(G) = \{ C_1, C_2, \dots, C_s \}$ and $s \geq 1$.
By Theorem~\ref{0923-4},
we have $a_{2l-1} = -2\sum_{j = 1}^s \re (C_j)$.
In particular, there exists $f_{2l-1} \in \MB{Z}[x] \setminus \{0\}$ such that $a_{2l-1} = f_{2l-1}(\cos \theta)$.
\end{proof}

A complex number $\alpha$ is an {\it algebraic number} 
if there exists a polynomial $p \in \MB{Z}[x] \setminus \{0\}$ such that $p(\alpha) = 0$.
The numbers $1$, $\frac{2}{3}$ and $\sqrt{5}$ are algebraic numbers,
but it is well-known that both $\pi$ and $e$ are not algebraic numbers.
We denote the set of all algebraic numbers by $\bar{\MB{Q}}$.
It is well-known that $\bar{\MB{Q}}$ forms a field \cite{L}.
The following is essential to obtain our first main theorem.

\begin{thm} \label{1224-1}
{\it
Let $\theta \in (0, \pi]$.
If $\cos \theta$ is not an algebraic number,
then $\theta$ has the bipartite detection property.
}
\end{thm}

\begin{proof}
Let $G$ be a mixed graph with $n$ vertices,
and let $\Phi(x) = \sum_{j = 0}^n a_j x^{n-j}$ be the characteristic polynomial of $H_{\theta}(G)$.
By Lemma~\ref{0926-1},
it is sufficient to show that $G$ is bipartite if it has the $\theta$-symmetric spectrum.
We would like to show the contrapositive of this statement.
Suppose $G$ is non-bipartite.
By Lemma~\ref{0926-2},
there exist a positive integer $l$ and a polynomial $f_{2l-1} \in \MB{Z}[x] \setminus \{0\}$
such that $a_{2l-1} = f_{2l-1}(\cos \theta)$.
We assume that $G$ has the $\theta$-symmetric spectrum.
Then Lemma~\ref{0926-3} derives $a_{2l-1} = 0$, i.e., $f_{2l-1}(\cos \theta) = 0$.
This implies that $\cos \theta$ is an algebraic number, which contradicts our assumption.
Therefore, $G$ does not have the $\theta$-symmetric spectrum.
\end{proof}

From the above theorem,
we only need to find angles whose cosines are not algebraic numbers.
This is achieved by borrowing some well-known theorems from transcendental number theory.
The first theorem we use is the Lindemann--Weierstrass theorem.


\begin{thm}[Theorem~1.4 in \cite{B}] \label{0126-1}
{\it
For any distinct algebraic numbers $\alpha_1, \dots, \alpha_n$ and
any non-zero algebraic number $\beta_1 \dots, \beta_n$,
we have $\beta_1 e^{\alpha_1} + \cdots + \beta_n e^{\alpha_n} \neq 0$.
}
\end{thm}

\begin{cor} \label{0926-4}
{\it
Let $\theta \in (0, \pi]$.
If $\theta \in \bar{\MB{Q}} \setminus \{0\}$,
then $\cos \theta$ is not an algebraic number.
In particular,
an angle $\theta \in \bar{\MB{Q}} \setminus \{0\}$ has the bipartite detection property.
}
\end{cor}

\begin{proof}
Suppose that $\cos \theta$ is an algebraic number.
Then $0$, $i \theta$ and $-i \theta$ are distinct algebraic numbers since $\theta \neq 0$.
However, we have $2 \cos \theta \cdot e^0 - e^{i \theta} - e^{-i \theta} = 0$,
which contradicts the Lindemann--Weierstrass theorem (Theorem~\ref{0126-1}).
Thus, $\cos \theta$ is not an algebraic number.
Furthermore, Theorem~\ref{1224-1} derives that an angle $\theta \in \bar{\MB{Q}} \setminus \{0\}$ has the bipartite detection property.
\end{proof}

The second theorem we use is the Gelfond--Schneider theorem,
but we insert one lemma before it.

\begin{lem} \label{1231-1}
{\it
Let $\theta \in (0, \pi]$.
If $\cos \theta$ is an algebraic number,
then so is $e^{i\theta}$.
}
\end{lem}

\begin{proof}
Since $\cos \theta$ is an algebraic number,
$1- \cos^2 \theta$ is also an algebraic number.
Thus, there exists a polynomial $f(x) \in \MB{Z}[x] \setminus \{0\}$ such that $f(1 - \cos^2 \theta) = 0$.
Define $g(x) = f(x^2)$.
We see that $g(\sin \theta) = g(\sqrt{1 - \cos^2 \theta}) = f(1 - \cos^2 \theta) = 0$.
Thus, $\sin \theta$ is an algebraic number,
and hence $e^{i \theta} = \cos \theta + i \sin \theta$ is also an algebraic number.
\end{proof}

\begin{thm}[\cite{L}] \label{0126-2}
{\it
If $\alpha$ is an algebraic number that is neither $0$ nor $1$
and if $\beta$ is algebraic irrational,  
then $\alpha^{\beta}$ is not an algebraic number.
}
\end{thm}

\begin{cor} \label{0115-1}
{\it
Let $\alpha \in (0, 1]$.
If $\alpha \in \bar{\MB{Q}} \setminus \MB{Q}$,
then $\cos (\alpha \pi)$ is not an algebraic number. 
In particular,
an angle $\theta =$ $\alpha \pi$ for $\alpha \in \bar{\MB{Q}} \setminus \MB{Q}$ has the bipartite detection property.
}
\end{cor}

\begin{proof}
We suppose that $\cos (\alpha \pi)$ is an algebraic number.
From Lemma~\ref{1231-1}, it follows that $e^{\alpha \pi i}$ is an algebraic number, whereas by Gelfond--Schneider theorem (Theorem~\ref{0126-2}),
$e^{\alpha \pi i} = (e^{\pi i})^{\alpha} = (-1)^{\alpha}$ is not an algebraic number.
This is a contradiction.
Therefore, $\cos (\alpha \pi)$ is not an algebraic number.
Furthermore, Theorem~\ref{1224-1} derives that an angle $\theta =$ $\alpha \pi$ for $\alpha \in \bar{\MB{Q}} \setminus \MB{Q}$ has the bipartite detection property.
\end{proof}

The third theorem we use is Baker's theorem.

\begin{thm}[Theorem~2.3 in \cite{B}] \label{1231-2}
{\it
$e^{\beta_0}\alpha_1^{\beta_1} \cdots \alpha_n^{\beta_n}$ is not an algebraic number for any non-zero algebraic numbers $\alpha_1, \dots, \alpha_n, \beta_1, \dots, \beta_n$.
}
\end{thm}

\begin{cor} \label{0115-2}
{\it
Let $\alpha, \beta$ be real numbers.
If $\alpha, \beta \in \MB{\bar{Q}} \setminus \{0\}$,
then $\cos(\alpha \pi + \beta)$ is not an algebraic number.
In particular,
an angle $\theta = \alpha \pi + \beta \in (0, \pi]$ for $\alpha, \beta \in \MB{\bar{Q}} \setminus \{0\}$ has the bipartite detection property.
}
\end{cor}

\begin{proof}
We suppose that $\cos (\alpha \pi + \beta)$ is an algebraic number.
From Lemma~\ref{1231-1}, it follows that $e^{(\alpha \pi + \beta) i}$ is an algebraic number, whereas by Theorem~\ref{1231-2} with $n=1$, $\beta_0 = i\beta$, $\alpha_1 = -1$, and $\beta_1 = \alpha$,
we have $e^{(\alpha \pi + \beta) i} = e^{i\beta}(e^{\pi i})^{\alpha} = e^{i\beta}(-1)^{\alpha}$ is not an algebraic number.
This is a contradiction.
Therefore, $\cos (\alpha \pi + \beta)$ is not an algebraic number.
Furthermore, Theorem~\ref{1224-1} derives that
an angle $\theta = \alpha \pi + \beta \in (0, \pi]$ for $\alpha, \beta \in \MB{\bar{Q}} \setminus \{0\}$ has the bipartite detection property.
\end{proof}

Putting together Corollary~\ref{0926-4},
Corollary~\ref{0115-1}, and Corollary~\ref{0115-2}
yields our first main theorem, Theorem~\ref{main_BDP} in Section~\ref{Intro}.

\section{Familiar angles} \label{1117-1}
In contrast to the previous section,
the bipartite detection property breaks down for angles familiar to us,
namely $\theta \in \MB{Q}\pi$.
To show this,
we would like to construct a counterexample
to the converse of Lemma~\ref{0926-1} for each $\theta \in \MB{Q}\pi$.
The following lemma shows that a proper mixed graph with the $\theta$-symmetric spectrum can be constructed from an oriented graph with the $\theta$-symmetric spectrum:

\begin{lem} \label{1203-2}
{\it
Let $\theta \in (0, \pi]$.
If a non-bipartite oriented graph has the $\theta$-symmetric spectrum,
then there exists a non-bipartite proper mixed graph with the $\theta$-symmetric spectrum.
}
\end{lem}

\begin{proof}
Let $G$ be a non-bipartite oriented graph with the $\theta$-symmetric spectrum,
and let $n$ be the number of vertices of $G$.
The matrix
\[
\begin{bmatrix}
H_{\theta}(G) & I_n \\
I_n & H_{\theta}(G)
\end{bmatrix}
\]
defines a proper mixed graph $\tilde{G}$.
As can be seen immediately,
$\Spec(\tilde{G}; \theta) = \{ \lambda \pm 1 \mid \lambda \in \Spec(G; \theta)\}$.
Since $G$ has the $\theta$-symmetric spectrum,
the mixed graph $\tilde{G}$ also has the $\theta$-symmetric spectrum.
We note that $\tilde{G}$ has $G$ as a subgraph.
Since $G$ is non-bipartite,
the mixed graph $\tilde{G}$ is also non-bipartite.
\end{proof}

Thus, we wish to construct non-bipartite oriented graphs with the $\theta$-symmetric spectrum for given $\theta \in \MB{Q}\pi$.
\color{black}

Let $m$ be a positive integer, and let $a \in \{1, \dots, m-1 \}$.
Define the oriented path $P_m^{(a, m-1-a)} = (V, \MC{A})$ with $m$ vertices by
$V = \{ p_1, p_2, \dots, p_m \}$ and
\[ \MC{A} = \{ (p_1, p_2), \dots, (p_a, p_{a+1}), (p_{a+2}, p_{a+1}), (p_{a+3}, p_{a+2}), \cdots, (p_m, p_{m-1}) \}. \]
We have illustrated $P_6^{(3,2)}$ in Figure~\ref{0930-2} as an example.

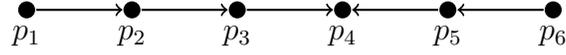
\begin{figure}[ht]
\begin{center}
\begin{tikzpicture}
[scale = 0.7,
line width = 0.8pt,
v/.style = {circle, fill = black, inner sep = 0.8mm},u/.style = {circle, fill = white, inner sep = 0.1mm}]
  \node[u] (L1) at (0, -0.5) {$p_1$};
  \node[u] (L1) at (2, -0.5) {$p_2$};
  \node[u] (L1) at (4, -0.5) {$p_3$};
  \node[u] (L1) at (6, -0.5) {$p_4$};
  \node[u] (L1) at (8, -0.5) {$p_5$};
  \node[u] (L1) at (10, -0.5) {$p_6$};
  \node[v] (1) at (0, 0) {};
  \node[v] (2) at (2, 0) {};
  \node[v] (3) at (4, 0) {};
  \node[v] (4) at (6, 0) {};
  \node[v] (5) at (8, 0) {};
  \node[v] (6) at (10, 0) {};  
  \draw[->] (1) to (2);
  \draw[->] (2) to (3); 
  \draw[->] (3) to (4);
  \draw[->] (5) to (4);
  \draw[->] (6) to (5);
\end{tikzpicture}
\caption{The mixed path $P_6^{(3,2)}$} \label{0930-2}
\end{center}
\end{figure}

Next, we construct the oriented graph $G_m$, which is key to this section.

\begin{enumerate}[Step 1.]
\item Take the join of the undirected path 
$P_{2}= (\{x, y\}, \{\{x,y\}\})$ with $2$ vertices and
the empty graph $\overline{K_m} = (\{z_1, \dots, z_m\}, \emptyset )$.
Here, we refer to \cite{BH} for the {\it join} of graphs $\G$ and $\Delta$.
An example of this step for $m=4$ is shown in Figure~\ref{0930-s1}.
\item Assign the orientation from the vertex $x$ to the vertex $y$.
\item Replace the edge $\{ x, z_j \}$ with $P_m^{(0, m-1)}$
so that $p_1 = x$ and $p_m = z_j$ for each $j$.
The process up to this step for $m=4$ is shown in Figure~\ref{1001-s3}.
\item Finally, replace the edge $\{ y, z_j \}$ with $P_m^{(j-1, m-j)}$
so that $p_1 = y$ and $p_m = z_j$ for each $j$.
Let $G_m$ be the resulting oriented graph.
As an example, we show the oriented graph $G_4$ in Figure~\ref{1001-s4}.
\end{enumerate}

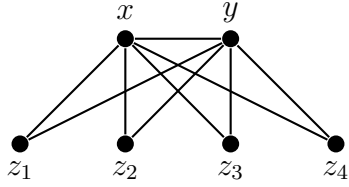
\begin{figure}[ht]
\begin{center}
\begin{tikzpicture}
[scale = 0.7,
line width = 0.8pt,
v/.style = {circle, fill = black, inner sep = 0.8mm},u/.style = {circle, fill = white, inner sep = 0.1mm}]
  \node[u] (Lx) at (-1, 0.5) {$x$};
  \node[u] (Ly) at (1, 0.5) {$y$};
  \node[u] (Lz1) at (-3, -2.5) {$z_1$};
  \node[u] (Lz2) at (-1, -2.5) {$z_2$};
  \node[u] (Lz3) at (1, -2.5) {$z_3$};
  \node[u] (Lz4) at (3, -2.5) {$z_4$};
  \node[v] (x) at (-1, 0) {};
  \node[v] (y) at (1, 0) {};
  \node[v] (z1) at (-3, -2) {};
  \node[v] (z2) at (-1, -2) {};
  \node[v] (z3) at (1, -2) {};
  \node[v] (z4) at (3, -2) {};  
  \draw[-] (x) to (y);
  \draw[-] (x) to (z1);
  \draw[-] (x) to (z2);
  \draw[-] (x) to (z3);
  \draw[-] (x) to (z4);
  \draw[-] (y) to (z1);
  \draw[-] (y) to (z2);
  \draw[-] (y) to (z3);
  \draw[-] (y) to (z4);
\end{tikzpicture}
\caption{The join of $P_2$ and $\overline{K_4}$} \label{0930-s1}
\end{center}
\end{figure}

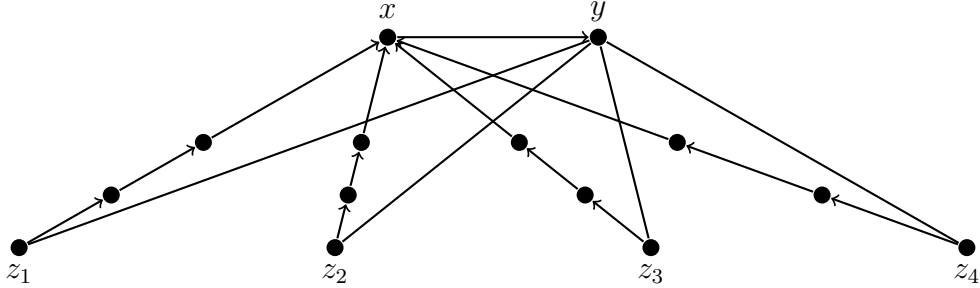
\begin{figure}[ht]
\begin{center}
\begin{tikzpicture}
[scale = 0.7,
line width = 0.8pt,
v/.style = {circle, fill = black, inner sep = 0.8mm},u/.style = {circle, fill = white, inner sep = 0.1mm}]
\node[u] (Lx) at (-2, 0.5) {$x$};
\node[u] (Ly) at (2, 0.5) {$y$};
\node[u] (Lz1) at (-9, -4.5) {$z_1$};
\node[u] (Lz2) at (-3, -4.5) {$z_2$};
\node[u] (Lz3) at (3, -4.5) {$z_3$};
\node[u] (Lz4) at (9, -4.5) {$z_4$};
\node[v] (x) at (-2, 0) {};
\node[v] (y) at (2, 0) {};
\node[v] (z1) at (-9, -4) {};
\node[v] (z1x) at (-5.5, -2) {};
\node[v] (z1x2) at (-7.25, -3) {};
%
\node[v] (z2) at (-3, -4) {};
\node[v] (z2x) at (-2.5, -2) {};
\node[v] (z2x2) at (-2.75, -3) {};
%
\node[v] (z3) at (3, -4) {};
\node[v] (z3x) at (0.5, -2) {};
\node[v] (z3x2) at (1.75, -3) {};
%
\node[v] (z4) at (9, -4) {};  
\node[v] (z4x) at (3.5, -2) {};
\node[v] (z4x2) at (6.25, -3) {};
%
\draw[->] (x) to (y);
\draw[->] (z1) to (z1x2);
\draw[->] (z1x2) to (z1x);
\draw[->] (z1x) to (x);
\draw[-] (y) to (z1);
\draw[->] (z2) to (z2x2);
\draw[->] (z2x2) to (z2x);
\draw[->] (z2x) to (x);
\draw[-] (y) to (z2);
\draw[->] (z3) to (z3x2);
\draw[->] (z3x2) to (z3x);
\draw[->] (z3x) to (x);
\draw[-] (y) to (z3);
\draw[->] (z4) to (z4x2);
\draw[->] (z4x2) to (z4x);
\draw[->] (z4x) to (x);
\draw[-] (y) to (z4);
\end{tikzpicture}
\caption{The mixed graph up to Step~3} \label{1001-s3}
\end{center}
\end{figure}

\begin{figure}[ht]
\begin{center}
\begin{tikzpicture}
[scale = 0.7,
line width = 0.8pt,
v/.style = {circle, fill = black, inner sep = 0.8mm},u/.style = {circle, fill = white, inner sep = 0.1mm}]
\node[u] (Lx) at (-2, 0.5) {$x$};
\node[u] (Ly) at (2, 0.5) {$y$};
\node[u] (Lz1) at (-9, -4.5) {$z_1$};
\node[u] (Lz2) at (-3, -4.5) {$z_2$};
\node[u] (Lz3) at (3, -4.5) {$z_3$};
\node[u] (Lz4) at (9, -4.5) {$z_4$};
\node[v] (x) at (-2, 0) {};
\node[v] (y) at (2, 0) {};
\node[v] (z1) at (-9, -4) {};
\node[v] (z1x) at (-5.5, -2) {};
\node[v] (z1x2) at (-7.25, -3) {};
\node[v] (z1y) at (-3.5, -2) {};
\node[v] (z1y2) at (-6.25, -3) {};
\node[v] (z2) at (-3, -4) {};
\node[v] (z2x) at (-2.5, -2) {};
\node[v] (z2x2) at (-2.75, -3) {};
\node[v] (z2y) at (-0.5, -2) {};
\node[v] (z2y2) at (-1.75, -3) {};
\node[v] (z3) at (3, -4) {};
\node[v] (z3x) at (0.5, -2) {};
\node[v] (z3x2) at (1.75, -3) {};
\node[v] (z3y) at (2.5, -2) {};
\node[v] (z3y2) at (2.75, -3) {};
\node[v] (z4) at (9, -4) {};  
\node[v] (z4x) at (3.5, -2) {};
\node[v] (z4x2) at (6.25, -3) {};
\node[v] (z4y) at (5.5, -2) {};
\node[v] (z4y2) at (7.25, -3) {};
\draw[->] (x) to (y);
\draw[->] (z1) to (z1x2);
\draw[->] (z1x2) to (z1x);
\draw[->] (z1x) to (x);
\draw[->] (z1) to (z1y2);
\draw[->] (z1y2) to (z1y);
\draw[->] (z1y) to (y);
\draw[->] (z2) to (z2x2);
\draw[->] (z2x2) to (z2x);
\draw[->] (z2x) to (x);
\draw[->] (z2) to (z2y2);
\draw[->] (z2y2) to (z2y);
\draw[<-, blue] (z2y) to (y);
\draw[->] (z3) to (z3x2);
\draw[->] (z3x2) to (z3x);
\draw[->] (z3x) to (x);
\draw[->] (z3) to (z3y2);
\draw[<-, blue] (z3y2) to (z3y);
\draw[<-, blue] (z3y) to (y);
\draw[->] (z4) to (z4x2);
\draw[->] (z4x2) to (z4x);
\draw[->] (z4x) to (x);
\draw[<-, blue] (z4) to (z4y2);
\draw[<-, blue] (z4y2) to (z4y);
\draw[<-, blue] (z4y) to (y);
\end{tikzpicture}
\caption{The oriented graph $G_4$} \label{1001-s4}
\end{center}
\end{figure}
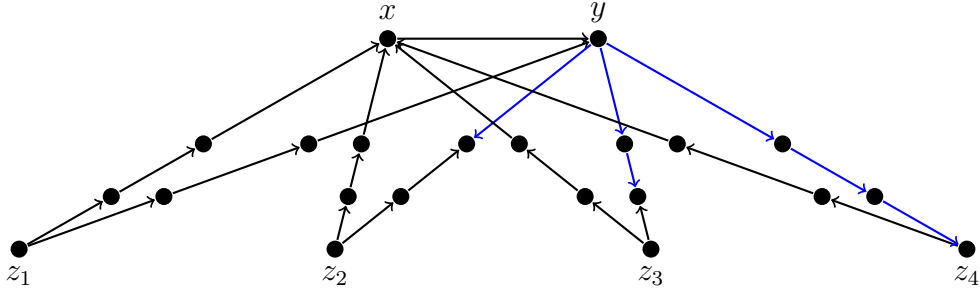

We have a few additional comments on the oriented graph $G_m$.
The underlying graph 
$\Gamma(G_{m})$ 
is a kind of book graph, that is, 
a graph obtained from $m$ copies 
of $C_{2m-1}$ by sharing a common 
edge $\{ x, y \}$. 
Moreover, in terms of magnetic flux 
in Section~\ref{MAM}, 
the flux of the cycle of length $2m-1$ passing through $z_{j}$ 
in $G_m$ is $\pm(2j-1)\theta$.
The number of vertices of $G_m$ is $2m^2-3m+2$.

There are two important properties of the oriented graph $G_m$.
The first is that the length of any odd cycle in $G_m$ is $2m-1$.
The second is that even if any odd cycle is removed from $G_m$,
the underlying graphs of the remaining oriented graphs are all isomorphic to each other.
These two properties lead to the following important lemma.

\begin{lem} \label{0930-1}
{\it
Let $\Phi(x) = \sum_{j = 0}^{2m^2 - 3m + 2} a_j x^{2m^2 - 3m + 2-j}$ be
the characteristic polynomial of $H_{\theta}(G_m)$ for $\theta \in (0, \pi]$.
Then,
\begin{enumerate}[(i)]
\item For any $k < m$, we have $a_{2k-1} = 0$.
\item If $a_{2m-1} = 0$, then we have $a_{2k-1} = 0$ for any $k > m$.
\end{enumerate}
}
\end{lem}

\begin{proof}
(i) Let $k < m$.
An elementary subgraph with $2k-1$ vertices must have at least one odd cycle.
However, the length of the minimum odd cycles in $G_m$ is $2m-1$,
so $\MC{H}_{2k-1}(G_m) = \emptyset$.
This implies that $a_{2k-1} = 0$.

(ii) Let $k > m$.
An elementary subgraph with $2k -1$ vertices must have odd cycles of length $2m-1$,
while any odd cycle in $G_m$ has the arc $(x,y)$.
Thus, an elementary subgraph with $2k -1$ vertices consists of
precisely one odd cycle of length $2m-1$ and disjoint union of $k-m$ arcs.
Let $C_j$ be the odd cycle in $G_m$ of length $2m-1$ that contains the vertex $z_j$.
Then $\G(G_m \setminus C_j)$ is isomorphic to $(m-1)P_{2m-3}$ for any $j$,
where $(m-1)P_{2m-3}$ denotes the disjoint union of $m-1$ paths $P_{2m-3}$.
Then, let $M_k$ be the number of elementary subgraphs in $(m-1)P_{2m-3}$ with $2(k-m)$ vertices,
i.e., $M_k = |\MC{H}_{2(k-m)}((m-1)P_{2m-3})|$.
We would like to show that $a_{2k-1} = (-1)^{k-m}  M_k  a_{2m-1}$.
By Theorem~\ref{0923-4},
\begin{align*}
a_{2k-1} &= \sum_{H \in \MC{H}_{2k-1}(G_m)} (-1)^{p(H)}2^{|\MC{C(H)|}} \prod_{C \in \MC{C}(H)} \re(C) \\
&= \sum_{j = 1}^m \sum_{\substack{H \in \MC{H}_{2k-1}(G_m) \\ H = C_j \cup (G_m \setminus C_j)}} (-1)^{p(H)}2^{|\MC{C(H)|}} \prod_{C \in \MC{C}(H)} \re(C) \\
&= \sum_{j = 1}^m \sum_{\substack{H \in \MC{H}_{2k-1}(G_m) \\ H = C_j \cup (G_m \setminus C_j)}} (-1)^{k-m+1}2^{1} \re(C_j) \\
&= \sum_{j = 1}^m \sum_{H \in \MC{H}_{2(k-m)}(G_m \setminus C_j)} (-1)^{k-m+1}2^{1} \re(C_j) \\
&= (-1)^{k-m} \sum_{j = 1}^m (-1)^1 2^1 \re(C_j) \sum_{H \in \MC{H}_{2(k-m)}(G_m \setminus C_j)} 1 \\
&= (-1)^{k-m} \sum_{j = 1}^m (-1)^1 2^1 \re(C_j) \sum_{H \in \MC{H}_{2(k-m)}((m-1)P_{2m-3})} 1 \\
&= (-1)^{k-m} \sum_{j = 1}^m (-1)^1 2^1 \re(C_j) M_k \\
&= (-1)^{k-m}  M_k  \sum_{j = 1}^m (-1)^1 2^1 \re(C_j) \\
&= (-1)^{k-m}  M_k  a_{2m-1}.
\end{align*}
Therefore, we see that $a_{2k-1} = 0$ if $a_{2m-1} = 0$.
\end{proof}

Before proving our second main theorem,
we have two brief additions for the case $\theta = \pi$.
It can be seen substantially via sign-symmetric signed graphs that $\theta = \pi$ does not have the bipartite detection property.
Interested readers in them can refer to \cite{GHMM, S} for example.
Of course, it is also possible to construct an example to break down the bipartite detection property
without notion of sign-symmetric singed graphs,
such as the mixed graph shown in Figure~\ref{1201-1}.
The other addition is that no oriented graph can provide examples to break down the bipartite detection property when $\theta = \pi$.
This is because $\Spec(G; \pi) = -\Spec(\G(G); 0)$ holds for any oriented graph $G$
since $H_{\pi}(G) = -A(\G(G))$,
where $A(\G(G))$ is the adjacency matrix of $\G(G)$.
In this sense,
if mixed graphs under consideration are restricted to oriented graphs,
we could say that {\it $\theta = \pi$ has the bipartite detection property for oriented graphs}.

\begin{figure}[ht]
\begin{center}
\begin{tikzpicture}
[scale = 0.7,
line width = 0.8pt,
v/.style = {circle, fill = black, inner sep = 0.8mm},u/.style = {circle, fill = white, inner sep = 0.1mm}]
  \node[v] (1) at (0, 0) {};
  \node[v] (2) at (2, 0) {};
  \node[v] (3) at (2, 2) {};
  \node[v] (4) at (0, 2) {};
  \draw[->] (1) to (2);
  \draw[-] (2) to (3); 
  \draw[-] (3) to (4);
  \draw[-] (4) to (1);
  \draw[-] (1) to (3);
\end{tikzpicture}
\caption{Non-bipartite mixed graph that has the $\pi$-symmetric spectrum} \label{1201-1}
\end{center}
\end{figure}
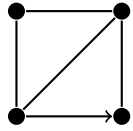

Now, we provide our second main theorem,
Theorem~\ref{main_notBDP} in Section~\ref{Intro}.

\begin{thm} \label{1119-2}
{\it
Let $\theta \in (0, \pi]$.
If $\theta \in \MB{Q}\pi$,
then $\theta$ does not have the bipartite detection property.
}
\end{thm}

\begin{proof}
Let $\theta = \frac{l}{m}\pi$.
The mixed graph in Figure~\ref{1201-1} shows that
$\theta = \pi$ does not have the bipartite detection property.
Then we can assume that $0 < \frac{l}{m} < 1$.
We show the non-bipartite oriented graph $G_m$ has the $\theta$-symmetric spectrum.
Under the same setting as Lemma~\ref{0930-1},
we wish to show that $a_{2m-1} = 0$.
For $j \in \{1, \dots, m\}$,
let $C_j$ be the odd cycle in $G_m$ of length $2m-1$ that contains the vertex $z_j$.
By Theorem~\ref{0923-4},
we have
\begin{align*}
a_{2m-1} &= \sum_{j = 1}^m (-1)^1 2^1 \re (C_j) \\
&= - \sum_{j = 1}^m 2\cos (2j-1)\theta \\
&= - \sum_{j = 1}^m (e^{(2j-1)i \theta} + e^{-(2j-1)i \theta}) \\
&= - \left( \frac{e^{i\theta}(e^{2mi\theta} - 1) }{e^{2i\theta} - 1}  + \frac{e^{-i\theta}(e^{-2mi\theta} - 1) }{e^{-2i\theta} - 1} \right) \\ 
&= - \left( \frac{e^{i\theta}(e^{2 l \pi i} - 1) }{e^{2i\theta} - 1}  + \frac{e^{-i\theta}(e^{-2 l \pi i} - 1) }{e^{-2i\theta} - 1} \right) \\
&= 0.
\end{align*}
Here we remark that $e^{2i\theta} \neq 1$.
Therefore, 
Lemma~\ref{0930-1} and Lemma~\ref{0926-3} derive that $G_m$ has the $\theta$-symmetric spectrum.
\end{proof}

\section{
Minimality of our oriented book graph} \label{1222-1}
In this section,
we discuss certain minimality of oriented graphs that break down the bipartite detection property.
Let us consider $\theta = \frac{\pi}{3}$ as an example.
According to Theorem~\ref{1119-2},
the oriented graph $G_3$ is an example that breaks down the bipartite detection property for this angle.
$G_3$ has $11$ vertices and its (longest) length of odd cycles is 5.
On the other hand,
Mohar \cite{M} constructed an oriented graph with 5 vertices and its longest length of odd cycles is 3 (See Figure~\ref{1124-1}).
In this way, our oriented graph $G_m$ is sufficient to investigate the bipartite detection property,
but there is still room for further investigation in terms of certain minimality.

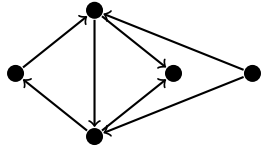
\begin{figure}[ht]
\begin{center}
\begin{tikzpicture}
[scale = 0.7,
line width = 0.8pt,
v/.style = {circle, fill = black, inner sep = 0.8mm},u/.style = {circle, fill = white, inner sep = 0.1mm}]
  \node[v] (1) at (0, 0) {};
  \node[v] (2) at (1.5, 1.2) {};
  \node[v] (3) at (1.5, -1.2) {};
  \node[v] (4) at (3, 0) {};
  \node[v] (5) at (4.5, 0) {};
  \draw[->] (1) to (2);
  \draw[->] (3) to (1);
  \draw[->] (2) to (3); 
  \draw[->] (2) to (4);
  \draw[->] (3) to (4);
  \draw[->] (5) to (2);
  \draw[->] (5) to (3);
\end{tikzpicture}
\caption{The non-bipartite oriented graph provided by Mohar \cite{M}
that has the $\frac{\pi}{3}$-symmetric spectrum} \label{1124-1}
\end{center}
\end{figure}

Let $\G$ be an undirected graph that is not a tree.
The {\it circumference} of $\G$ is the length of the longest cycles in $\G$,
and the {\it odd circumference} of $\G$ is the length of the longest odd cycles in $\G$.
These two terms are carried over for mixed graphs, that is,
the {\it circumference} of a mixed graph $G$ is the length of the longest cycles in $\G(G)$.
The same for {\it odd circumference} of $G$.
For example, the oriented graph shown in Figure~\ref{1124-1} has the circumference $4$,
while the odd circumference is $3$.
To conclude first,
we have succeeded in lowering the odd circumferences
depending on the denominator of the angle $\theta \in \MB{Q}\pi$,
as summarized in Table~\ref{1119-3} on page  \pageref{1119-3}.

To construct oriented graphs with the symmetric spectrum and smaller odd circumference, we first generalize the oriented graph $G_m$.

\begin{enumerate}[Step 1.]
\item Take the join of the undirected path $P_2 = (\{x, y\}, \{\{x,y\}\})$ with $2$ vertices and
the empty graph $\overline{K_{s_1+ \dots + s_t}} = (\{ z_1^{(1)}, \dots, z_1^{(s_1)}, z_2^{(1)}, \dots, z_2^{(s_2)}, \dots, z_t^{(1)}, \dots, z_t^{(s_t)} \}, \emptyset )$.
An example of this step for $(s_1, s_2, s_3) = (2,2,1)$ is shown in Figure~\ref{1125-s1}.
\item Assign the orientation from the vertex $x$ to the vertex $y$.
\item Replace the edge $\{ x, z_j^{(k)} \}$ with $P_t^{(0, t-1)}$
so that $p_1 = x$ and $p_t = z_j^{(k)}$ for each $j, k$.
\item Finally, replace the edge $\{ y, z_j^{(k)} \}$ with $P_t^{(j-1, t-j)}$
so that $p_1 = y$ and $p_m = z_j^{(k)}$ for each $j, k$.
Let $G(s_1, s_2, \dots, s_t)$ be the resulting oriented graph.
As an example, we show the oriented graph $G(2,2,1)$ in Figure~\ref{1125-s2}.
\end{enumerate}

\begin{figure}[ht]
\begin{center}
\begin{tikzpicture}
[scale = 0.7,
line width = 0.8pt,
v/.style = {circle, fill = black, inner sep = 0.8mm},u/.style = {circle, fill = white, inner sep = 0.1mm}]
  \node[u] (Lx) at (-1, 0.5) {$x$};
  \node[u] (Ly) at (1, 0.5) {$y$};
  \node[u] (Lz11) at (-3, -2.7) {$z_1^{(1)}$};
  \node[u] (Lz12) at (-2, -2.7) {$z_1^{(2)}$};
  \node[u] (Lz21) at (0, -2.7) {$z_2^{(1)}$};
  \node[u] (Lz22) at (1, -2.7) {$z_2^{(2)}$};
  \node[u] (Lz31) at (3, -2.7) {$z_3^{(1)}$};
  \node[v] (x) at (-1, 0) {};
  \node[v] (y) at (1, 0) {};
  \node[v] (z11) at (-3, -2) {};
  \node[v] (z12) at (-2, -2) {};
  \node[v] (z21) at (0, -2) {};
  \node[v] (z22) at (1, -2) {};
  \node[v] (z31) at (3,-2) {};
  \draw[-] (x) to (y);
  \draw[-] (x) to (z11);
  \draw[-] (x) to (z12);
  \draw[-] (x) to (z21);
  \draw[-] (x) to (z22);
  \draw[-] (x) to (z31);
  \draw[-] (y) to (z11);
  \draw[-] (y) to (z12);
  \draw[-] (y) to (z21);
  \draw[-] (y) to (z22);
  \draw[-] (y) to (z31);
\end{tikzpicture}
\caption{The join of $P_2$ and $\overline{K_{2+2+1}}$} \label{1125-s1}
\end{center}
\end{figure}
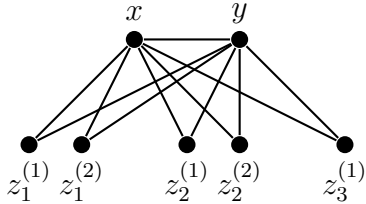

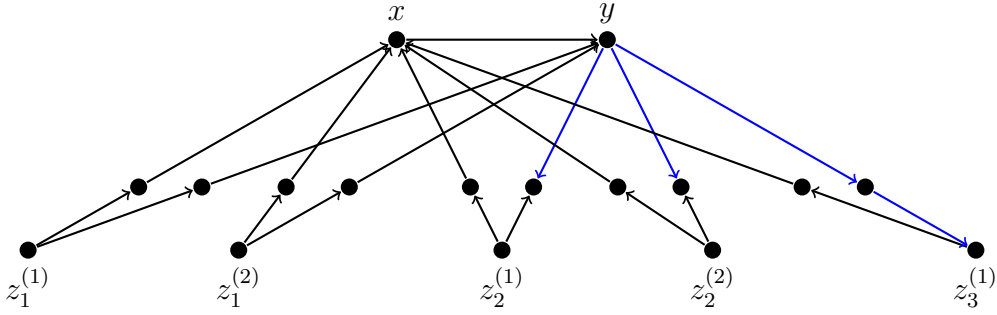
\begin{figure}[ht]
\begin{center}
\begin{tikzpicture}
[scale = 0.7,
line width = 0.8pt,
v/.style = {circle, fill = black, inner sep = 0.8mm},u/.style = {circle, fill = white, inner sep = 0.1mm}]
\node[u] (Lx) at (-2, 0.5) {$x$};
\node[u] (Ly) at (2, 0.5) {$y$};
\node[u] (Lz11) at (-9, -4.7) {$z_1^{(1)}$};
\node[u] (Lz12) at (-5, -4.7) {$z_1^{(2)}$};
\node[u] (Lz21) at (0, -4.7) {$z_2^{(1)}$};
\node[u] (Lz22) at (4, -4.7) {$z_2^{(2)}$};
\node[u] (Lz31) at (9, -4.7) {$z_3^{(1)}$};
\node[v] (x) at (-2, 0) {};
\node[v] (y) at (2, 0) {};
\node[v] (z11) at (-9, -4) {};
\node[v] (z11sx) at ($(z11)!0.3!(x)$) {};
\node[v] (z11sy) at ($(z11)!0.3!(y)$) {};
\node[v] (z12) at (-5, -4) {};
\node[v] (z12sx) at ($(z12)!0.3!(x)$) {};
\node[v] (z12sy) at ($(z12)!0.3!(y)$) {};
\node[v] (z21) at (0, -4) {};
\node[v] (z21sx) at ($(z21)!0.3!(x)$) {};
\node[v] (z21sy) at ($(z21)!0.3!(y)$) {};
\node[v] (z22) at (4, -4) {};  
\node[v] (z22sx) at ($(z22)!0.3!(x)$) {};
\node[v] (z22sy) at ($(z22)!0.3!(y)$) {};
\node[v] (z31) at (9, -4) {};
\node[v] (z31sx) at ($(z31)!0.3!(x)$) {};
\node[v] (z31sy) at ($(z31)!0.3!(y)$) {};
\draw[->] (x) to (y);
\draw[->] (z11) to (z11sx);
\draw[->] (z11sx) to (x);
\draw[->] (z11) to (z11sy);
\draw[->] (z11sy) to (y);
\draw[->] (z12) to (z12sx);
\draw[->] (z12sx) to (x);
\draw[->] (z12) to (z12sy);
\draw[->] (z12sy) to (y);
\draw[->] (z21) to (z21sx);
\draw[->] (z21sx) to (x);
\draw[->] (z21) to (z21sy);
\draw[<-, blue] (z21sy) to (y);
\draw[->] (z22) to (z22sx);
\draw[->] (z22sx) to (x);
\draw[->] (z22) to (z22sy);
\draw[<-, blue] (z22sy) to (y);
\draw[->] (z31) to (z31sx);
\draw[->] (z31sx) to (x);
\draw[<-, blue] (z31) to (z31sy);
\draw[<-, blue] (z31sy) to (y);
\end{tikzpicture}
\caption{The oriented graph $G(2,2,1)$} \label{1125-s2}
\end{center}
\end{figure}

Naturally, the underlying graph $\Gamma(G(s_1, s_2, \dots, s_t))$ is also a kind of book graph with $\sum_{h=1}^{t}s_{h}$ sheets, that is, 
a graph obtained from $\sum_{h=1}^{t}s_{h}$ copies of $C_{2t-1}$ by sharing a common edge $\{x, y\}$. 
Moreover, in terms of magnetic flux 
in Section~\ref{MAM}, 
the flux of the cycle of length $2t-1$ passing through $z^{(k)}_{j}$ 
in $G(s_1, s_2, \dots, s_t)$ is $\pm(2j-1)\theta$.

The odd circumference of $G(s_1, s_2, \dots, s_t)$ is easily captured by looking at the number of arguments.
Namely, the odd circumference of $G(s_1, s_2, \dots, s_t)$ is $2t-1$
since the number of arguments is $t$.
Of course, this oriented graph requires that $t \geq 2$ and
\begin{equation} \label{1120-3}
s_1 + s_2 + \dots + s_t \geq 1
\end{equation}
to make sense.

We note that our oriented graph $G(s_1, s_2, \dots, s_t)$ essentially contains the oriented graph in Figure~\ref{1124-1} provided by Mohar \cite{M}.
It can be seen that they share the same characteristic polynomial by deformation and switching as shown in Figure~\ref{1126-1}.
Remark that both of $H_{\theta}$'s are unitarily equivalent by
Theorem~\ref{UE}. 
For a precise definition of switching, see for example \cite{R}.
Readers who want to understand switching visually can also refer to Section~4.1 in \cite{KSY}.

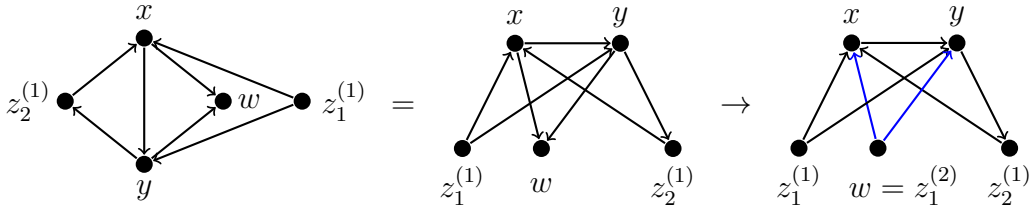
\begin{figure}[ht]
\begin{center}
\begin{tikzpicture}
[scale = 0.7,
line width = 0.8pt,
v/.style = {circle, fill = black, inner sep = 0.8mm},u/.style = {circle, fill = white, inner sep = 0.1mm}]
  \node[u] (d) at (0, -2.7) {};
  \node[u] (L1) at (-0.7, 0) {$z_2^{(1)}$};
  \node[u] (L2) at (1.5, 1.7) {$x$};
  \node[u] (L3) at (1.5, -1.7) {$y$};
  \node[u] (L4) at (3.5, 0) {$w$};
  \node[u] (L5) at (5.3, 0) {$z_1^{(1)}$};
  \node[v] (1) at (0, 0) {};
  \node[v] (2) at (1.5, 1.2) {};
  \node[v] (3) at (1.5, -1.2) {};
  \node[v] (4) at (3, 0) {};
  \node[v] (5) at (4.5, 0) {};
  \draw[->] (1) to (2);
  \draw[->] (3) to (1);
  \draw[->] (2) to (3); 
  \draw[->] (2) to (4);
  \draw[->] (3) to (4);
  \draw[->] (5) to (2);
  \draw[->] (5) to (3);
\end{tikzpicture}
\raisebox{50pt}{$\, = \,$}
\begin{tikzpicture}
[scale = 0.7,
line width = 0.8pt,
v/.style = {circle, fill = black, inner sep = 0.8mm},u/.style = {circle, fill = white, inner sep = 0.1mm}]
  \node[u] (d) at (0, -3.8) {};
  \node[u] (Lx) at (-1, 0.5) {$x$};
  \node[u] (Ly) at (1, 0.5) {$y$};
  \node[u] (Lz11) at (-2, -2.7) {$z_1^{(1)}$};
  \node[u] (Lw) at (-0.5, -2.7) {$w$};
  \node[u] (Lz21) at (2, -2.7) {$z_2^{(1)}$};
  \node[v] (x) at (-1, 0) {};
  \node[v] (y) at (1, 0) {};
  \node[v] (z11) at (-2, -2) {};
  \node[v] (w) at (-0.5, -2) {};
  \node[v] (z21) at (2, -2) {};
\draw[->] (x) to (y);
\draw[<-] (x) to (z11);
\draw[->] (x) to (w);
\draw[<-] (x) to (z21);
\draw[<-] (y) to (z11);
\draw[->] (y) to (w);
\draw[->] (y) to (z21);
\end{tikzpicture}
\raisebox{50pt}{$\, \to \,$}
\begin{tikzpicture}
[scale = 0.7,
line width = 0.8pt,
v/.style = {circle, fill = black, inner sep = 0.8mm},u/.style = {circle, fill = white, inner sep = 0.1mm}]
  \node[u] (Lx) at (-1, 0.5) {$x$};
  \node[u] (Ly) at (1, 0.5) {$y$};
  \node[u] (Lz11) at (-2, -2.7) {$z_1^{(1)}$};
  \node[u] (Lw) at (0, -2.7) {$w = z_1^{(2)}$};
  \node[u] (Lz21) at (2, -2.7) {$z_2^{(1)}$};
  \node[v] (x) at (-1, 0) {};
  \node[v] (y) at (1, 0) {};
  \node[v] (z11) at (-2, -2) {};
  \node[v] (w) at (-0.5, -2) {};
  \node[v] (z21) at (2, -2) {};
\draw[->] (x) to (y);
\draw[<-] (x) to (z11);
\draw[<-, blue] (x) to (w);
\draw[<-] (x) to (z21);
\draw[<-] (y) to (z11);
\draw[<-, blue] (y) to (w);
\draw[->] (y) to (z21);
\end{tikzpicture}
\caption{The oriented graph provided by Mohar \cite{M} is essentially $G(2,1)$} \label{1126-1}
\end{center}
\end{figure}

The oriented graph $G_m$ constructed in Section~\ref{1117-1} is nothing but \[ G(\underbrace{1,1,\dots, 1}_{m}). \]
In the arguments of $G(s_1, s_2, \dots, s_t)$,
the number of repetitions is expressed by superscripts if the same number is repeated in succession.
Under this notation, $G_m$ is described as $G(1^{(m)})$, i.e.,
\[ G_m = G(\underbrace{1,1,\dots, 1}_{m}) = G(1^{(m)}). \]
Note that the odd circumference of $G(1^{(m)})$ is $2m-1$.
Also, $G(0^{(n-1)}, 1)$ is the directed cycle with $2n-1$ vertices.

For our oriented book graph 
$G(s_{1},s_{2},\dots ,s_{t})$, 
its odd circumference is $2t-1$ 
and its number of sheets is $\sum_{j=1}^{t}s_{j}$.  
Then, for any positive integers $k$ and $h$, 
the odd circumference and  
the number of sheets of 
the book graph $G(ks_{1},ks_{2},\dots ,ks_{t},0^{(h)})$ 
are $2(t+h)-1$ and $k\sum_{j=1}^{t}s_{j}$, respectively. 
We should remark that, if $G(s_{1},s_{2},\dots ,s_{t})$ has the $\theta$-symmetric spectrum for some $\theta$, then $G(ks_{1},ks_{2},\dots ,ks_{t},0^{(h)})$ has also the $\theta$-symmetric spectrum for 
the same $\theta$. 
Therefore we have to discuss the minimality of the odd circumference and the number of sheets of 
our oriented book graph $G(s_{1},s_{2},\dots ,s_{t})$ for each $\theta$.
Hereinafter we give some graphs having the $\theta$-symmetric spectrum smaller than that in Section~\ref{1117-1}.

Let $\theta = \frac{l}{m}\pi \in \MB{Q}\pi$.
We say that the angle $\theta$ is {\it irreducible} if $m$ and $l$ are coprime.
If the denominator $m$ of an irreducible angle is congruent to 2 mod 4,
we have a non-bipartite oriented graph whose odd circumference is $\frac{m}{2}$
as an example that breaks down the bipartite detection property.
Remark that $\Gamma(G(0^{(t-1)}, 1))$ is isomorphic to $C_{2t-1}$.

\begin{pro} \label{m2}
{\it
Let $\theta = \frac{l}{m}\pi$ be an irreducible angle.
If $m \equiv 2 \pmod 4$,
then the non-bipartite oriented graph $C = G(0^{(\frac{m-2}{4})}, 1)$,
whose odd circumference is $\frac{m}{2}$,
has the $\theta$-symmetric spectrum.
}
\end{pro}
\begin{proof}
Let $\det(xI_{\frac{m}{2}} - H_{\theta}(C)) = \sum_{j = 0}^{\frac{m}{2}} a_j x^{\frac{m}{2} - j}$.
Since the girth of $C$ is $\frac{m}{2}$,
$a_k = 0$ for odd $k < \frac{m}{2}$.
On the other hand, Theorem~\ref{0921-1} derives
$a_{\frac{m}{2}} = (-1)^1 2^1 \re(C) = -2\cos \frac{m}{2} \theta = -2 \cos \frac{l}{2} \pi = 0$.
By Lemma~\ref{0926-3},
we see that $C$ has the $\theta$-symmetric spectrum.
\end{proof}

When we examine the case $m \not\equiv 2 \pmod 4$,
it is necessary to generalize Lemma~\ref{0930-1}.

\begin{lem} \label{1109-1}
{\it
Let $G$ be a mixed graph with $n$ vertices.
Assume that $G$ satisfies the following two conditions:
\begin{enumerate}[(a)]
\item 
There exists a positive integer $m$ such that the length of any odd cycle of $G$ is $2m-1$.
\item There exists a forest $F$ such that for any odd cycle $C$ of $G$, the graph $\G(G \setminus C)$ is isomorphic to $F$,
where $G \setminus C$ is the subgraph induced by $V(G) \setminus V(C)$.
\end{enumerate}
Let $\det(xI_{n} - H_{\theta}(G)) = \sum_{j = 0}^{n} a_j x^{n - j}$ for $\theta \in (0, \pi]$.
Then we have
\begin{enumerate}[(i)]
\item For any $k < m$, we have $a_{2k-1} = 0$.
\item If $a_{2m-1} = 0$, then we have $a_{2k-1} = 0$ for any $k > m$.
\end{enumerate}
}
\end{lem}

\begin{proof}
The essence of the proof is the same as in Lemma~\ref{0930-1}.

(i)
Let $k < m$.
An elementary subgraph with $2k-1$ vertices must have at least one odd cycle.
However, the length of any odd cycles in $G$ is $2m-1$ by the condition~(a).
Thus, we have $\MC{H}_{2k-1}(G) = \emptyset$,
i.e., $a_{2k-1} = 0$.

(ii)
We consider $k > m$.
Let $C_1, C_2, \dots, C_t$ be all odd cycles of length $2m-1$ in $G$,
and let $s$ be the number of connected components of $F$.
Define $M_k$ to be the number of elementary subgraphs in $F$ with $2(k-m)$ vertices, i.e.,
$M_k = |\MC{H}_{2(k-m)}(F)|$.
We would like to show that $a_{2k-1} = (-1)^s M_k a_{2m-1}$.
By Theorem~\ref{0923-4},
\begin{align*}
a_{2k-1} &= \sum_{H \in \MC{H}_{2k-1}(G)} (-1)^{p(H)}2^{|\MC{C}(H)|} \prod_{C \in \MC{C}(H)} \re(C) \\
&= \sum_{j = 1}^t \sum_{\substack{H \in \MC{H}_{2k-1}(G) \\ H = C_j \cup (G \setminus C_j)}} (-1)^{p(H)}2^{|\MC{C}(H)|} \prod_{C \in \MC{C}(H)} \re(C) \\
&= \sum_{j = 1}^t \sum_{\substack{H \in \MC{H}_{2k-1}(G) \\ H = C_j \cup (G \setminus C_j)}} (-1)^{s+1}2^{1} \re(C_j)  \tag{by (b)} \\
&= \sum_{j = 1}^t (-1)^{s+1} 2^1 \re(C_j) \sum_{H \in \MC{H}_{2(k-m)}(G \setminus C_j)} 1 \\
&= \sum_{j = 1}^t (-1)^{s+1} 2^1 \re(C_j) \sum_{H \in \MC{H}_{2(k-m)}(F)} 1 \tag{by (b)} \\
&= \sum_{j = 1}^t (-1)^{s+1} 2^1 \re(C_j) M_k \\
&= (-1)^s M_k \sum_{j = 1}^t (-1)^{1} 2^1 \re(C_j) \\
&= (-1)^s M_k a_{2m-1}.
\end{align*}
Therefore, we see that $a_{2k-1} = 0$ if $a_{2m-1} = 0$.
\end{proof}

\begin{pro} \label{m13}
{\it
Let $\theta = \frac{l}{m}\pi$ be an irreducible angle.
If $m$ is odd,
then the non-bipartite oriented graph $G(2^{(\frac{m-1}{2})}, 1)$,
whose odd circumference is $m$,
has the $\theta$-symmetric spectrum.
}
\end{pro}
\begin{proof}
Let $G = G(2^{(\frac{m-1}{2})}, 1)$.
The size of all odd cycles in $G$ is $m$.
Also, for any odd cycle $C$ in $G$,
the graph $\G(G \setminus C)$ is isomorphic to $(m-1) P_{m-2}$,
so $G$ is in the situation where Lemma~\ref{1109-1} applies.
Under the same notation as Lemma~\ref{1109-1},
we would like to show that $a_m = 0$.
Let $C_{j}^{(s)}$ denote the odd cycle containing the vertex $z_j^{(s)}$ in $G$.
Then we have $\MC{H}_m(G) = \{C_1^{(1)}, C_1^{(2)}, C_2^{(1)}, C_2^{(2)}, \dots, C_{\frac{m-1}{2}}^{(1)}, C_{\frac{m-1}{2}}^{(2)}, C_{\frac{m+1}{2}}^{(1)} \}$.
By Theorem~\ref{0923-4},
\begin{align*}
a_m &= \sum_{j=1}^{\frac{m+1}{2}} (-1)^1 2^1 \re(C_j^{(1)}) + \sum_{j=1}^{\frac{m-1}{2}} (-1)^1 2^1 \re(C_j^{(2)}) \\
&= -2 \left\{ \sum_{j=1}^{\frac{m-1}{2}} \cos (2j-1)\theta + \sum_{j=1}^{\frac{m+1}{2}} \cos (2j-1)\theta \right\} \\
&= -2 \left\{ \sum_{j=1}^{\frac{m-1}{2}} \cos (2j-1)\theta + \sum_{j=1}^{\frac{m+1}{2}} \cos (2m - (2j-1))\theta \right\} \\
&= -2 \sum_{j=1}^{m} \cos (2j-1)\theta \\
&= 0.
\end{align*}
By Lemma~\ref{1109-1} and Lemma~\ref{0926-3},
we see that $G$ has the $\theta$-symmetric spectrum.
\end{proof}

\begin{pro} \label{m0}
{\it
Let $\theta = \frac{l}{m}\pi$ be an irreducible angle.
If $m \equiv 0 \pmod 4$,
then the non-bipartite oriented graph $G(0^{(\frac{m}{4}-1)}, 1^{(2)})$,
whose odd circumference is $\frac{m}{2}+1$,
has the $\theta$-symmetric spectrum.
}
\end{pro}

\begin{proof}
Let $G = G(0^{(\frac{m}{4}-1)}, 1^{(2)})$.
We note that the size of any odd cycles in $G$ is $\frac{m}{2}+1$.
Also, for any odd cycle $C$ in $G$,
the graph $\G(G \setminus C)$ is isomorphic to $P_{\frac{m}{2}-1}$,
so $G$ is in the situation where Lemma~\ref{1109-1} applies.
Under the same notation as Lemma~\ref{1109-1},
we would like to show that $a_{\frac{m}{2}+1} = 0$.
Note that $l$ is odd since $m$ and $l$ are coprime.
We have
\begin{equation} \label{1119-1}
\cos \left( \frac{m}{2}+1 \right)\theta
= \cos \left( l \pi - \left( \frac{m}{2} -1 \right)\theta \right)
= - \cos \left( \frac{m}{2}-1 \right)\theta.
\end{equation}
By Theorem~\ref{0923-4},
\begin{align*}
a_{\frac{m}{2}+1} &= (-1)^1 2^1 \cos \left( \frac{m}{2} - 1 \right)\theta + (-1)^1 2^1 \cos \left( \frac{m}{2} + 1 \right)\theta \\
&= -2\left( \cos \left( \frac{m}{2} - 1 \right)\theta - \cos \left( \frac{m}{2} - 1 \right)\theta \right) \tag{by (\ref{1119-1})} \\
&= 0.
\end{align*}
By Lemma~\ref{1109-1} and Lemma~\ref{0926-3},
we see that $G$ has the $\theta$-symmetric spectrum.
\end{proof}

We summarize the discussion up to here.
We have shown in Theorem~\ref{1119-2} that $\theta \in \MB{Q}\pi$ does not have the bipartite detection property,
but the odd circumference of the oriented graph $G_m$ used there was $2m-1$.
In Proposition~\ref{m2}, Proposition~\ref{m13}, and Proposition~\ref{m0},
we attempted to reduce odd circumferences.
The results are summarized in Table~\ref{1119-3}.

%
\begin{table}[htbp]
  \centering
  \begin{tabular}{|c|c|c|c|} \hline
    $m$ & Odd circumference &
     \# of sheets & Proposition \\ \hline\hline
    $*$ & $ 2m-1$ & $m$ & Theorem~\ref{1119-2} \\ \hline
    odd & $m$ & $m$ &\ref{m13} \\
    $\equiv 2 \pmod 4$ & $\frac{m}{2}$ & $1$ &  \ref{m2}  \\
    $\equiv 0 \pmod 4$ & $\frac{m}{2} + 1$ & $2$ & \ref{m0} \\ \hline
  \end{tabular}
  \caption{Upper bounds of odd circumferences and the number of sheets of our oriented book graphs for each irreducible angle $\theta = \frac{l}{m}\pi$}
  \label{1119-3}
\end{table}
In the following,
lower bounds of odd circumferences are discussed.
From Proposition~\ref{m13},
we know that the upper bound of odd circumferences in non-bipartite oriented graphs with $\theta$-symmetric spectrum is $m$ for an irreducible angle $\theta = \frac{l}{m}\pi$.
In fact, if $m$ is odd prime,
the oriented graph $G(s_1, \dots, s_t)$ whose odd circumference is less than $m$ does not have the $\theta$-symmetric spectrum:

\begin{pro} \label{1122-4}
{\it
Let $\theta = \frac{l}{p}\pi$ be an irreducible angle,
where $p$ is odd prime.
Then, $G(s_1, \dots, s_{\frac{p-1}{2}})$,
whose odd circumference is $p-2$,
does not have the $\theta$-symmetric spectrum.
}
\end{pro}

\begin{proof}
Let $G = G(s_1, \dots, s_{\frac{p-1}{2}})$.
We would like to derive a contradiction by assuming that $G$ has the $\theta$-symmetric spectrum.
Let $\Psi(x) = \sum_{j=0}^n a_j x^{n-j}$ be the characteristic polynomial of $H_{\theta}(G)$,
where $n$ is the number of vertices of $G$.
Since $G$ has the $\theta$-symmetric spectrum,
we have
\begin{equation} \label{1120-1}
a_{p-2} = 0
\end{equation}
by Lemma~\ref{0926-3}.
Let $\zeta_k = e^{\frac{2\pi}{k}i}$ for a positive integer $k$.
Theorem~\ref{0923-4} derives
\begin{align}\label{1218-1}
a_{p-2} &= -2 \sum_{j=1}^{\frac{p-1}{2}} s_j \cos (2j-1)\theta \notag\\
&= -\sum_{j=1}^{\frac{p-1}{2}}s_{j}(\zeta_{2p}^{(2j-1)l} + \zeta_{2p}^{-(2j-1)l}) \notag\\
&= -\sum_{j=1}^{\frac{p-1}{2}}s_j( \zeta_{2}^l \zeta_{2}^l \zeta_{2p}^{(2j-1)l} + \zeta_{2}^l \zeta_{2}^l \zeta_{2p}^{-(2j-1)l}) \notag\\
&= -\sum_{j=1}^{\frac{p-1}{2}}s_j( (-1)^l \zeta_{2p}^{pl} \zeta_{2p}^{(2j-1)l} + (-1)^l \zeta_{2p}^{pl} \zeta_{2p}^{-(2j-1)l}) \notag\\
&= (-1)^{l+1}\sum_{j=1}^{\frac{p-1}{2}}s_j(\zeta_{2p}^{(p+2j-1)l} + \zeta_{2p}^{(p-2j+1)l}) \notag\\
&= (-1)^{l+1}\sum_{j=1}^{\frac{p-1}{2}}s_j(\zeta_{p}^{\frac{(p+2j-1)l}{2}} + \zeta_{p}^{\frac{(p-2j+1)l}{2}}). 
\end{align}
Here we should remark that $p$ is odd.
Combining with Equality~(\ref{1120-1}) yields
\begin{equation} \label{1120-2}
\sum_{j=1}^{\frac{p-1}{2}}s_j(\zeta_{p}^{\frac{(p+2j-1)l}{2}} + \zeta_{p}^{\frac{(p-2j+1)l}{2}}) = 0.
\end{equation}
In addition,
$\{ \zeta_{p}^{\frac{(p+2j-1)l}{2}} \mid j = 1, 2, \dots, \frac{p-1}{2} \} \cup \{ \zeta_{p}^{\frac{(p-2j+1)l}{2}} \mid j = 1, 2, \dots, \frac{p-1}{2} \}
= \{ \zeta_{p}^{jl} \mid j = 1, 2, \dots, p-1 \}$.
Thus, Equality~(\ref{1120-2}) is a $\MB{Q}$-linear combination of $\zeta_p^l, \zeta_p^{2l}, \dots, \zeta_p^{(p-1)l}$.
However, the degree of the $\MB{Q}$-minimal polynomial of $\zeta_p^l$ is $p-1$,
so the $\MB{Q}$-linear combination has to be trivial.
Thus, $s_1 = s_2 = \dots = s_{\frac{p-1}{2}} = 0$.
This contradicts~(\ref{1120-3}).
\end{proof}

The above proposition is an assertion on minimality with respect to the odd circumference,
whereas in fact it is also possible to show minimality with respect to the number of sheets:

\begin{pro} \label{1219-1}
{\it
Let $\theta = \frac{l}{p}\pi$ be an irreducible angle,
where $p$ is odd prime.
If the oriented graph $G(s_1, \dots, s_t)$ has the $\theta$-symmetric spectrum, then $s_1 + \dots + s_t \geq p$.
The equality holds if $(s_1, \dots, s_{\frac{p-1}{2}}, s_{\frac{p+1}{2}}) = (2,\dots, 2,1)$.
}
\end{pro}

\begin{proof}
First, we would like to check that it is sufficient to investigate only the case where the number of arguments of $G(s_1, \dots, s_t)$ is $\frac{p+1}{2}$.
Suppose $t > p$.
For arbitrary positive integer $k$ and for $j \in \{1,2,\dots,p\}$,
the oriented graphs $G(s_1, \dots, s_j, \dots, s_{kp+j}, \dots, s_t)$ and $G(s_1, \dots, s_j + 1, \dots, s_{kp+j} -1, \dots, s_t)$ share the same characteristic polynomial and the same number of sheets, i.e. the same sum of arguments.
This is because $\cos (2(kp+j)-1)\theta = \cos (2j-1)\theta$.
Thus, we may assume that $t \leq p$.
In addition, we suppose $\frac{p+1}{2} < t \leq p$.
For $j \in \{1,2,\dots,\frac{p-1}{2}\}$,
the oriented graphs $G(s_1, \dots, s_{\frac{p+1}{2} - j}, \dots, s_{\frac{p+1}{2} + j}, \dots, s_t)$ and $G(s_1, \dots, s_{\frac{p+1}{2} - j} + 1, \dots, s_{\frac{p+1}{2} + j} -1, \dots, s_t)$ share the same characteristic polynomial and the same number of sheets.
This is because $\cos (2(\frac{p+1}{2} - j)-1)\theta = \cos (2(\frac{p+1}{2} + j)-1)\theta$.
Thus, we may assume that $t \leq \frac{p+1}{2}$.
On the other hand,
Proposition~\ref{1122-4} implies that $G(s_1, \dots, s_t)$ does not have the $\theta$-symmetric spectrum if $t < \frac{p+1}{2}$.
Therefore, it is sufficient to consider $t = \frac{p+1}{2}$,
and we have
\begin{equation} \label{1130-3}
s_{\frac{p+1}{2}} \geq 1.
\end{equation}

Next, we show that if the oriented graph $G = G(s_1, \dots, s_{\frac{p+1}{2}})$ has the $\theta$-symmetric spectrum, then $s_1 + \dots + s_{\frac{p+1}{2}} \geq p$.
Since $G$ has the $\theta$-symmetric spectrum,
we have
\begin{equation} \label{1130-1}
a_{p} = 0
\end{equation}
by Lemma~\ref{0926-3}.
On the other hand, Theorem~\ref{0923-4} with (\ref{1218-1}) derives
\begin{equation} \label{1130-2}
a_{p} = (-1)^{l+1}\sum_{j=1}^{\frac{p+1}{2}}s_j(\zeta_{p}^{\frac{(p+2j-1)l}{2}} + \zeta_{p}^{\frac{(p-2j+1)l}{2}}),
\end{equation}
where $\zeta_p = e^{\frac{2\pi}{p}i}$. 
By Equalities~(\ref{1130-1}) and~(\ref{1130-2}),
we have
\[
s_{\frac{p-1}{2}}\zeta_{p}^{(p-1)l} + s_{\frac{p-3}{2}}\zeta_{p}^{(p-2)l} + \dots + 
s_{1}\zeta_{p}^{\frac{p+1}{2}l} +
s_{1}\zeta_{p}^{\frac{p-1}{2}l} +
s_{2}\zeta_{p}^{\frac{p-3}{2}l} + \dots +
s_{\frac{p-1}{2}}\zeta_{p}^{l} +
2s_{\frac{p+1}{2}} = 0.
\]
%
Define the polynomial
\[ f(x) = s_{\frac{p-1}{2}}x^{p-1} + s_{\frac{p-3}{2}}x^{p-2} + \dots + 
s_{1}x^{\frac{p+1}{2}} +
s_{1}x^{\frac{p-1}{2}} +
s_{2}\zeta_{p}^{\frac{p-3}{2}} +
\dots +
s_{\frac{p-1}{2}}x +
2s_{\frac{p+1}{2}}. \]
This is a $\MB{Q}$-coefficient polynomial with degree at most $p-1$ satisfying $f(\zeta_p^l) = 0$.
Thus, the polynomial $f(x)$ is a constant multiple of the $\MB{Q}$-minimal polynomial of $\zeta_p^l$, i.e.,
there exists a rational number $k$ such that $f(x) = k(x^{p-1} + x^{p-2} + \dots + 1 )$.
We have $s_{\frac{p-1}{2}} = s_{\frac{p-3}{2}} = \dots = s_1 = 2s_{\frac{p+1}{2}} = k$.
This equality and (\ref{1130-3}) derive $1 \leq s_{\frac{p+1}{2}} = \frac{k}{2}$, that is, $k \geq 2$.
Therefore, we have
$s_1 + \dots + s_{\frac{p-1}{2}} + s_{\frac{p+1}{2}} = k \cdot \frac{p-1}{2} + \frac{k}{2} \geq p$.

Finally, we note that it has already been proved in Proposition~\ref{m13} that $G$ has the $\theta$-symmetric spectrum when $(s_1, \dots, s_{\frac{p-1}{2}}, s_{\frac{p+1}{2}}) = (2,\dots, 2,1)$, completing the proof.
\end{proof}

We return the discussion to the minimality of odd circumferences.
Table~\ref{1228-1} summarizes the minimum odd circumferences of our book graphs known at this stage in small $m$ for an irreducible angle $\theta = \frac{l}{m}\pi$.
Note that although we have only discussed the minimality of odd circumferences for odd primes in Proposition~\ref{1122-4},
we have also found the minimality for $m$ whose upper bound is $3$,
i.e., $m=4,6$.

\begin{table}[ht]
  \centering
  \begin{tabular}{|c||c|c|c|c|c|c|c|c|c|c|c|c} \hline
$m$ &3&4&5&6&7&8&9&10&11&12&13&$\cdots$ \\ \hline
Value in Table~\ref{1119-3} &3&3&5&3&7&5&9&5&11&7&13&$\cdots$ \\ \hline
Min. odd circumference &3&3&5&3&7&?&?&?&11&?&13&$\cdots$ \\ \hline
  \end{tabular}
  \caption{Minimum odd circumferences in small $m$ for an irreducible angle $\theta = \frac{l}{m}\pi$}
  \label{1228-1}
\end{table}

In the remainder of this paper,
we would like to discuss the minimum odd circumferences for $m \in \{8,9,10,12 \}$.
Relatively easy cases are $m = 8, 10$.
In these cases, values of cosines can be calculated explicitly,
and hence the same procedure as in Proposition~\ref{1122-4} can be shown that
the oriented graph $G(s_1, s_2)$,
whose odd circumference is $3$,
does not have the $\theta$-symmetric spectrum.
Thus, the odd circumferences are at least $5$,
while Table~\ref{1119-3} shows that this value is optimal.

For $m = 9$,
it is difficult to explicitly calculate the values of cosines,
so we use field theoretical techniques:

\begin{pro}
{\it
Let $\theta = \frac{l}{9}\pi$ be an irreducible angle.
Then $G(1,0,1,1)$ 
has the $\theta$-symmetric spectrum,
while $G(s_1, s_2, s_3)$ 
does not.
}
\end{pro}

\begin{proof}
Let us consider whether $G(s_{1},\dots ,s_{k})$ has the $\theta$-symmetric 
spectrum for $k=3$ or $4$; the odd circumference is $5$ or $7$, respectively. 
Let $\zeta_{18} = e^{\frac{2\pi}{18}i}$ and $\zeta_{9} = e^{\frac{2\pi}{9}i}$. 
We only have to show whether $a_{2k-1}=0$ or not. By the same arguments 
as in the proof of Proposition~\ref{1122-4}, we have 
\begin{align}\label{YuH1220-1}
a_{2k-1} &= -2\sum_{j=1}^{k} s_{j} \cos (2j-1)\theta \notag \\
&= -\sum_{j=1}^{k}s_{j} (\zeta_{18}^{(2j-1)l} + \zeta_{18}^{-(2j-1)l}) \notag\\
&=(-1)^{l+1}\zeta_{18}^{9l}
\sum_{j=1}^{k}s_{j} (\zeta_{18}^{(2j-1)l} + \zeta_{18}^{-(2j-1)l}) \notag\\
&=(-1)^{l+1}\sum_{j=1}^{k}s_{j} ((\zeta_{9}^{l})^{4+j} + (\zeta_{9}^{l})^{5-j}) 
\end{align}
Put $x=\zeta_{9}^{l}$ and 
$g(x) = \sum_{j=1}^{k}s_{j} (x^{4+j} + x^{5-j})$. 
Then $a_{2k-1}=0$ if and only if $g(x)=0$.
Note that, since $x=\zeta_{9}^{l}$ is a $9$-th primitive root of the unity, 
it satisfies $f(x)=0$, 
where $f(x)=x^{6} + x^{3} +1$ is the $9$-th cyclotomic polynomial. 
Using the relations $x^6 = -x^3-1$, $x^7 = -x^4 -x$ and $x^8 = -x^5 -x^2$, 
 we have, for $k=4$, 
\begin{equation}\label{YuH1220-2}
g(x) = (s_{1}-s_{4})x^{5}+(s_{1}-s_{3})x^{4}+(s_{3}-s_{4})x^{2}+
(-s_{3}+s_{4})x -s_{2}. 
\end{equation} 
Since $f(x)$ is the $\MB{Q}$-minimal polynomial of $\zeta_9$, 
$g(x)=0$ if and only if $s_{1}=s_{3}=s_{4}$ and $s_{2}=0$.
In particular, choosing $s_{1}=s_{3}=s_{4}=1$ and $s_{2}=0$, 
we find $G(1,0,1,1)$, whose odd circumference is $7$, has 
the $\theta$-symmetric spectrum.
On the other hand, for $k=3$, putting $s_{4}=0$ in (\ref{YuH1220-2}), 
we have 
\begin{equation}\label{YuH1220-3}
g(x) = s_{1}x^{5}+(s_{1}-s_{3})x^{4}+s_{3}x^{2}- s_{3}x -s_{2}.
\end{equation} 
Thus $g(x)=0$ if and only if $s_{1}=s_{2}=s_{3}=0$. 
It implies that any $G(s_{1},s_{2},s_{3})$, whose odd circumference is $5$, 
does not have the $\theta$-symmetric spectrum. 
\end{proof}

The last remaining case is $m = 12$.
So far, we have effectively used field theoretical techniques to show that
our oriented book graphs with small odd circumferences do not have the $\theta$-symmetric spectra.
However, this approach is not always valid.
An example to puzzle us appears in this case.

\begin{pro} \label{1203-1}
{\it
Let $\theta = \frac{l}{12}\pi$ be an irreducible angle.
Then, $G(s_1, s_2, s_3)$,
whose odd circumference is $5$,
does not have the $\theta$-symmetric spectrum.
}
\end{pro}

\begin{proof}
For simplicity, we show the case $l=1$,
but our proof is exactly the same for other $l$.
Let $\theta = \frac{1}{12}\pi$ and $G = G(s_1, s_2, s_3)$.
We would like to derive a contradiction by assuming that $G$ has the $\theta$-symmetric spectrum.
Let $\Psi(x) = \sum_{j=0}^n a_j x^{n-j}$ be the characteristic polynomial of $H_{\theta}(G)$,
where $n$ is the number of vertices of $G$.
Since $G$ has the $\theta$-symmetric spectrum,
we have
\begin{equation} \label{1122-1}
a_{5} = 0
\end{equation}
by Lemma~\ref{0926-3}.
On the other hand, from Theorem~\ref{0923-4}, we have
\begin{align*}
a_5 &= -2(s_1 \cos \theta + s_2 \cos 3\theta + s_3 \cos 5 \theta) \\
&= -2\left( s_1 \cos \frac{\pi}{12} + s_2 \cos \frac{3}{12}\pi + s_3 \cos \frac{5}{12}\pi \right) \\
&= -2\left( s_1 \frac{\sqrt{3}+1}{2\sqrt{2}} + s_2 \frac{1}{\sqrt{2}} + s_3 \frac{\sqrt{3}-1}{2\sqrt{2}} \right).
\end{align*}
Combining with Equality~(\ref{1122-1}) yields
$s_1 + 2s_2 - s_3 + (s_1 + s_3)\sqrt{3} = 0$.
Since $1$ and $\sqrt{3}$ are $\MB{Q}$-linearly independent, we have
\begin{align}
s_1 + 2s_2 - s_3 &= 0, \label{1122-3} \\
s_1 + s_3 &= 0. \label{1122-2}
\end{align}
Equality~(\ref{1122-2}) and $s_1, s_3 \geq 0$ lead to $s_1 = s_3 = 0$.
From this and Equality~(\ref{1122-3}), we have $s_2 = 0$.
This contradicts~(\ref{1120-3}).
\end{proof}

An interesting point in the above proof is that
even though $\cos \frac{\pi}{12}, \cos \frac{3}{12}\pi, \cos \frac{5}{12}\pi$ are $\MB{Q}$-linearly dependent,
we obtain $s_1 = s_2 = s_3 = 0$ because $s_1, s_2, s_3$ are non-negative.
This implies that strategies based only on field theoretical techniques do not generally provide lower bounds of odd circumferences.

In summary, Table~\ref{1228-1} is updated as shown in Table~\ref{1230-1}.
The values in Table~\ref{1119-3} are optimal for small $m$ except for $m=9$.
On the other hand, for $m=9$,
the optimal value is $7$ while the value in Table~\ref{1119-3} is $9$.
Table~\ref{1119-3} is not perfect, but not bad either.

\begin{table}[ht]
  \centering
  \begin{tabular}{|c||c|c|c|c|c|c|c|c|c|c|c|c} \hline
$m$ &3&4&5&6&7&8&9&10&11&12&13&$\cdots$ \\ \hline
Value in Table~\ref{1119-3} &3&3&5&3&7&5&9&5&11&7&13&$\cdots$ \\ \hline
Min. odd circumference &3&3&5&3&7&5&7&5&11&7&13&$\cdots$ \\ \hline
  \end{tabular}
  \caption{Minimum odd circumferences in small $m$ for an irreducible angle $\theta = \frac{l}{m}\pi$}
  \label{1230-1}
\end{table}

\section*{Acknowledgements}
We would like to thank Dr.~Yuta Nozaki for providing us some information on transcendental number theory.
Yu.H. acknowledges financial supports from the Grant-in-Aid of Scientific Research (C) Japan Society for the Promotion of Science (Grant No. 18K03401).
S.K. is supported by JSPS KAKENHI (Grant No. 20J01175).
E.S. acknowledges financial supports from the Grant-in-Aid of Scientific Research (C) Japan Society for the Promotion of Science (Grant No. 19K03616) and Research Origin for Dressed Photon.

\end{document}